\documentclass[twoside,a4paper,reqno,11pt]{amsart} 
\usepackage[top=28mm,right=28mm,bottom=28mm,left=28mm]{geometry}

\usepackage{microtype}
\usepackage{amsfonts, amsmath, amssymb, amsbsy, mathrsfs, bm, latexsym, stmaryrd, hyperref, array, enumitem, textcomp, url}
\usepackage[table]{xcolor}

\newcommand{\GL}{\operatorname{GL}}

\renewcommand{\a}{\alpha}
\renewcommand{\b}{\beta}

\newcommand{\e}{\epsilon}

\renewcommand{\O}{\Omega}

\newcommand{\la}{\langle}
\newcommand{\ra}{\rangle}
\newcommand{\leqs}{\leqslant}
\newcommand{\geqs}{\geqslant}
\newcommand{\normeq}{\trianglelefteqslant}

\newcommand{\vs}{\vspace{3mm}}

\makeatletter
\newcommand{\imod}[1]{\allowbreak\mkern4mu({\operator@font mod}\,\,#1)}
\makeatother

\theoremstyle{plain}
\newtheorem{theorem}{Theorem} 
 
\newtheorem{corol}[theorem]{Corollary}
\newtheorem{thm}{Theorem}[section] 
\newtheorem{lem}[thm]{Lemma}
\newtheorem{prop}[thm]{Proposition} 

\newtheorem{cor}[thm]{Corollary} 
\newtheorem*{theorem*}{Theorem} 
\newtheorem*{conj*}{Conjecture}

\theoremstyle{definition}
\newtheorem{rem}[thm]{Remark}
\newtheorem{defn}[thm]{Definition}
\newtheorem*{deff}{Definition}

\begin{document}

\title{Strongly base-two groups}
\dedicatory{\rm Dedicated to our friend and colleague Pham Huu Tiep on the occasion of his 60th birthday}

\author{Timothy C. Burness}
\address{T.C. Burness, School of Mathematics, University of Bristol, Bristol BS8 1UG, UK}
\email{t.burness@bristol.ac.uk}

\author{Robert M. Guralnick}
\address{R.M. Guralnick, Department of Mathematics, University of Southern California, Los Angeles, CA 90089-2532, USA}
\email{guralnic@usc.edu}

\date{\today} 

\begin{abstract}
Let $G$ be a finite group, let $H$ be a core-free subgroup and let $b(G,H)$ denote the base size for the action of $G$ on $G/H$. Let $\alpha(G)$ be the number of conjugacy classes of core-free subgroups $H$ of $G$ with $b(G,H) \geqslant 3$. We say that $G$ is a strongly base-two group if $\alpha(G) \leqslant 1$, which means that almost every faithful transitive permutation representation of $G$ has base size $2$. In this paper we study the strongly base-two finite groups with trivial Frattini subgroup.
\end{abstract}

\maketitle

\section{Introduction}\label{s:intro}

Let $G \leqs {\rm Sym}(\O)$ be a transitive permutation group on a finite set $\O$ with point stabilizer $H$. A subset $B$ of $\O$ is a \emph{base} for $G$ if the pointwise stabilizer of $B$ in $G$ is trivial. We write $b(G,H)$ for the \emph{base size} of $G$, which is the minimal size of a base for $G$. Here $H$ is a core-free subgroup of $G$ and $b(G,H)$ is the smallest positive integer $b$ with the property that
\[
\bigcap_{i=1}^b H^{g_i} = 1
\]
for some $g_i \in G$. We say that $G$ is a \emph{base-two} permutation group if $b(G,H) = 2$.

The base size of a finite permutation group is a classical invariant that has been studied for many decades, finding a wide range of applications and connections to a diverse range of problems in group theory and related areas (we refer the reader to the survey articles \cite{BC,LSh3} and \cite[Section 5]{Bur181} for more details). Determining the precise base size of a given group is a difficult problem and there has been an intense focus on establishing strong bounds on the base sizes of primitive groups (here each point stabilizer is a maximal subgroup). In this direction, there have been several major advances in recent years, partly motivated by a circle of highly influential conjectures due to Babai, Cameron, Kantor and Pyber, all of which have now been resolved.

There has also been substantial interest in understanding the base-two primitive groups, which is an ambitious project initiated by Jan Saxl in the 1990s. Although a complete classification remains out of reach, there has been significant progress towards this goal, which is best described in terms of the Aschbacher-O'Nan-Scott division of the finite primitive groups. For example, there are strong results of Fawcett \cite{Faw0,Faw} on diagonal groups and twisted wreath products, and work of several authors  \cite{F1,F2,Lee1,Lee2} on affine groups with quasisimple point stabilizers. The study of base-two product-type groups was recently initiated in \cite{BH2} and there is an expanding literature for almost simple groups. For instance, the base-two almost simple primitive groups with solvable point stabilizers are determined in \cite{B20} and there is a complete classification for the groups with socle an alternating or sporadic simple group (see \cite{BGS,BOW,James}). Some of these results have already found some interesting applications. For example, see \cite{BH_gen} for applications concerning strong $2$-generation properties of almost simple groups and \cite{BL,BTh} on the classification of extremely primitive groups. Base-two results also arise in \cite{BLN} in relation to bounding the diameter of the solvable graph of an almost simple group (here the vertices are labelled by the nontrivial elements of the group, with two vertices adjacent if they generate a solvable group). 

In this paper we introduce and study the following family of groups.

\begin{deff}\label{d:sb2}
Let $G$ be a finite group and let $\a(G)$ be the number of core-free subgroups $H$ of $G$ (up to conjugacy in $G$) such that $b(G,H) \geqs 3$ with respect to the natural action of $G$ on $G/H$. We call $G$ a \emph{strongly base-two} group if $\a(G) \leqs 1$. 
\end{deff}

For example, if $G = \mathbb{M}$ is the Monster sporadic simple group and $H$ is a proper subgroup, then \cite[Theorem 3.1]{B_spor} states that $b(G,H) \leqs 3$, with equality if and only if $H = 2.\mathbb{B}$ is the centralizer of a \texttt{2A} involution. It follows that $\a(G) = 1$ and thus the Monster is a strongly base-two group. It is therefore natural to ask if other simple (or almost simple) groups share this property and we will answer this question in Theorem \ref{t:asb2} below. More generally, our main aim is to study the strongly base-two finite groups with trivial Frattini subgroup. 

\begin{deff}\label{d:large}
Let $G$ be a finite group and let $H$ be a subgroup. We say that $H$ is \emph{large} if $H \ne G$ and $|H|^2 > |G|$.
\end{deff}

As a special case of Proposition \ref{p:large} below, every finite nonabelian simple group $T$ has a large subgroup. As an immediate consequence, we deduce that $\a(T) \geqs 1$ and this partly explains why we adopt the condition $\a(G) \leqs 1$ in the definition of a strongly base-two group, rather than the more restrictive property $\a(G)=0$. This observation also leads us naturally to the following definition, which we will need in order to state some of our main results.

\begin{deff}\label{d:special}
A finite nonabelian simple group $T$ is \emph{special} if it has a unique ${\rm Aut}(T)$-class of large subgroups.
\end{deff}

First we classify the special simple groups and we use this to determine all the strongly base-two almost simple groups. In Table \ref{tab:special}, we also identify a large subgroup $H$ of $T$ (note that $H$ is a Borel subgroup of $T$ in the first three rows of the table).

\begin{theorem}\label{t:special}
Let $T$ be a finite nonabelian simple group. Then $T$ is special if and only if $T$ is one of the groups recorded in Table \ref{tab:special}.
\end{theorem}

\begin{table}
\[
\begin{array}{lll} \hline
T & H & \mbox{Conditions} \\ \hline
{\rm L}_2(q) & [q]{:}(q-1)/d & \mbox{$d=(2,q-1)$, $(q-1)/d \geqs 7$ is a prime} \\
{}^2B_2(q) & [q^2]{:}(q-1) & \mbox{$q-1$ is a prime} \\
{\rm L}_{2}(3^5) & 3^5{:}121 & \\
{\rm J}_1 & {\rm L}_{2}(11) & \\
{\rm J}_3 & {\rm L}_2(16){:}2 & \\
\mathbb{M} & 2.\mathbb{B} & \\ \hline
\end{array}
\]
\caption{The special simple groups}
\label{tab:special}
\end{table}

\begin{theorem}\label{t:asb2}
Let $G$ be a finite almost simple group with socle $T$. Then $\a(G) \geqs 1$, with equality if and only if $G = T$ is one of the following special simple groups:
\begin{itemize}\addtolength{\itemsep}{0.2\baselineskip}
\item[{\rm (i)}] $G = {\rm L}_2(q)$, $q \geqs 23$ is odd and $(q-1)/2$ is a prime.
\item[{\rm (ii)}] $G = {}^2B_2(q)$ and $q-1$ is a prime.
\item[{\rm (iii)}] $G = {\rm J}_1$ or $\mathbb{M}$.
\end{itemize}
\end{theorem}

With Theorem \ref{t:asb2} in hand, we now extend our analysis to a wider class of finite groups. Initially, we were interested in the groups which admit a faithful primitive permutation representation (that is to say, the groups with a core-free maximal subgroup). But it turns out that the larger class of groups $G$ with $\Phi(G) = 1$ can be studied with similar methods, where $\Phi(G)$ is the Frattini subgroup of $G$. However, if we relax this condition on $\Phi(G)$, then the problem of understanding the core-free subgroups of $G$ becomes much more complicated and we do not pursue this here. 

In the results presented below, we use $F(G)$ and $F^*(G) = E(G)F(G)$ to denote the Fitting subgroup and the generalized Fitting subgroup of $G$, respectively, where $E(G)$ is the layer of $G$ (the latter is the product of the subnormal quasisimple subgroups of $G$). First we handle the groups with $E(G) \ne 1$. Note that the possibilities for $T$ are determined in Theorem \ref{t:asb2}.

\begin{theorem}\label{t:main1}
Let $G$ be a finite group such that $\Phi(G) = 1$ and $E(G) \ne 1$. Then $\a(G) \geqs 1$, with equality if and only if 
\[
G = T \times C_{p_1} \times \cdots \times C_{p_k},
\]
where $E(G)=T$ is simple, $\a(T)=1$, $k \geqs 0$ and each $p_i$ is a prime.
\end{theorem}

Now assume $\Phi(G) = E(G)=1$, so $F^*(G) = F(G)$ is abelian. In order to state our main results in this setting, we need to introduce some additional notation. Fix a prime $p$ and note that $P = O_p(G)$ is elementary abelian since $\Phi(G)=1$, so we may view $P$ as a $G$-module. Let $c_p(G)$ be the maximal dimension of a cyclic $G$-submodule of $P$ and let $c(G) \geqs 1$ be the maximum value of $c_p(G)$ over all primes $p$. Then a concise version of our main result is as follows.

\begin{theorem}\label{t:EG}
Let $G$ be a finite group such that $\Phi(G) = 1$ and $E(G)=1$.
\begin{itemize}\addtolength{\itemsep}{0.2\baselineskip}
\item[{\rm (i)}] If $c(G) \geqs 5$ then $\a(G) \geqs 2$.
\item[{\rm (ii)}] If $c(G)=4$ (respectively, $3$) then $\a(G) \geqs 1$, with equality if and only if $G$ satisfies all of the conditions in Theorem \ref{t:c4} (respectively, Theorem \ref{t:c3}).
\item[{\rm (iii)}] If $c(G) = 1$ then $\a(G) = 0$.
\end{itemize}
\end{theorem}

For parts (i) and (iii), we refer the reader to Theorems \ref{t:c5} and \ref{t:c1}, respectively. The analysis of the groups with $c(G)=2$ is more complicated and this explains why they are excluded in the statement of Theorem \ref{t:EG}. Some partial results in this special case are presented in Section \ref{s:cle2} and we note that there are examples with $\a(G) = 0$ and $\a(G) = 1$.

We conclude by recording two corollaries of our main results. First we classify all the nonsolvable strongly base-two finite groups with $\Phi(G)=1$. 

\begin{corol}\label{c:ns}
Let $G$ be a finite nonsolvable group with $\Phi(G)=1$. Then $\a(G) \geqs 1$, with equality if and only if one of the following holds, where $k \geqs 0$ and each $p_i$ is a prime:
\begin{itemize}\addtolength{\itemsep}{0.2\baselineskip}
\item[{\rm (i)}] $G = T \times C_{p_1} \times \cdots \times C_{p_k}$, where $T$ is a simple group with $\a(T) = 1$.
\item[{\rm (ii)}] $G = (C_{59})^2{:}{\rm SL}_2(5) \times C_{p_1} \times \cdots \times C_{p_k}$, where $(C_{59})^2{:}{\rm SL}_2(5) < {\rm AGL}_2(59)$ is primitive and $p_i \ne 2,59$ for all $i$.
\end{itemize}
\end{corol}

Our final result describes the strongly base-two groups with a core-free maximal subgroup. That is to say, we determine the strongly base-two finite groups which admit a faithful primitive permutation representation. In part (iii), recall that an irreducible linear group $L \leqs {\rm GL}_d(p)$ is \emph{imprimitive} if it preserves a nontrivial direct sum decomposition of the natural module $(\mathbb{F}_p)^d$, otherwise $L$ is \emph{primitive}.

\begin{corol}\label{c:prim}
Let $G$ be a finite group with socle $N$ and assume $G$ has a core-free maximal subgroup. Then $\a(G) \leqs 1$ if and only if one of the following holds:
\begin{itemize}\addtolength{\itemsep}{0.2\baselineskip}
\item[{\rm (i)}] $G= N$ is simple and $\a(G) = 1$. 
\item[{\rm (ii)}] $G = N{:}C_{p^2+p+1} < {\rm AGL}_3(p)$, where $N = (C_p)^3$ and $p \not\equiv 1 \imod{3}$.
\item[{\rm (iii)}] $G = N{:}L \leqs {\rm AGL}_2(p)$, where $N = (C_p)^2$ and one of the following holds:

\vspace{1mm}

\begin{itemize}\addtolength{\itemsep}{0.2\baselineskip}
\item[{\rm (a)}] $p=2$ and $L = C_3$ or $S_3$.
\item[{\rm (b)}] $p \geqs 3$, $L \cap Z({\rm GL}_2(p)) = 1$ and $L$ is either abelian or primitive.
\item[{\rm (c)}] $p \geqs 7$ and $L = D_{2r}$ is imprimitive, where $r$ is an odd prime divisor of $p-1$.
\item[{\rm (d)}]  $p \equiv 3 \imod{4}$, $|L|=2(p+1)$ and either $L$ is cyclic, or $Z(L) = C_2$ and $L$ is primitive.
\end{itemize}

\item[{\rm (iv)}] $G = N{:}L \leqs {\rm AGL}_1(p)$, where $N = C_p$ and $L \leqs {\rm GL}_1(p)$. 
\end{itemize}
\end{corol}

By inspecting the groups arising in Corollary \ref{c:prim}, we observe that $\a(G) \leqs 1$ only if $N$ is the unique minimal normal subgroup of $G$. Let us also record that $\a(G)=1$ in parts (i) and (ii), whereas  $\a(G) = 0$ for the groups in part (iv). In part (iii), we find that $\a(G) = 0$ if and only if $(p,L) = (2,C_3)$ or (b) holds. 

\vs

\noindent \textbf{Organization.} In Section \ref{s:special} we assume $G$ is an almost simple group and we prove Theorems \ref{t:special} and \ref{t:asb2}. For the remainder of the paper, we assume $G$ is an arbitrary finite group with $\Phi(G)=1$ and we begin our analysis by presenting some general lemmas in Section \ref{s:general}, which impose restrictions on the structure of the socle of a strongly base-two group. In Section \ref{s:t3} we prove Theorem \ref{t:main1}, which describes the strongly base-two groups with trivial Frattini subgroup and nonabelian socle. Next in Section \ref{s:special2} we handle a special family of groups with abelian socle, which is then used in our analysis of the groups with $E(G)=1$ and $c(G) \ne 2$ in Section \ref{s:c4}, where we prove Theorem \ref{t:EG}. Corollaries \ref{c:ns} and \ref{c:prim} are established in Section \ref{s:cor} and we finish by presenting some partial results on the strongly base-two groups with $c(G) = 2$ in 
Section \ref{s:cle2}.
 
\vs

\noindent \textbf{Notation.} Let $G$ be a finite group and let $n$ be a positive integer. We will write $C_n$, or just $n$, for a cyclic group of order $n$ and $G^n$ will denote the direct product of $n$ copies of $G$. An unspecified extension of $G$ by a group $H$ will be denoted by $G.H$; if the extension splits then we may write $G{:}H$. We use $[n]$ for an unspecified solvable group of order $n$. We adopt the standard notation for simple groups of Lie type from \cite{KL}.

\vs

\noindent \textbf{Acknowledgements.} We thank three anonymous referees for their helpful comments and suggestions on an earlier version of this paper. We also thank the Department of Mathematics at the California Institute of Technology for their generous hospitality during a research visit in spring 2022. Guralnick was partially supported by the NSF grant DMS-1901595 and a Simons Foundation Fellowship 609771. 

\section{Almost simple groups}\label{s:special}

In this section we prove Theorems \ref{t:special} and \ref{t:asb2}. We begin by recalling the following result, which is \cite[Theorem A]{FKL}.

\begin{thm}\label{t:fkl}
Let $G$ be a finite group of composite order. Then $G$ has a proper subgroup $H$ such that $|H|^2 \geqs |G|$.
\end{thm}

\begin{rem}
The proof of this result relies on the classification of finite simple groups. Indeed, by applying induction on $|G|$, it is straightforward to reduce the proof to nonabelian simple groups, at which point one can inspect each family of simple groups in turn and subsequently identify a maximal subgroup with the desired property. This extends earlier work of Brauer and Fowler \cite[Theorem 2D]{BF}, who showed that every finite group $G$ with even composite order has a proper subgroup $H$ with $|H|^3 \geqs |G|$ (if $|G|$ is odd and composite, then $G$ is solvable by the Feit-Thompson theorem and it is easy to identify a proper subgroup $H$ with $|H|^2 \geqs |G|$).  
\end{rem}

\begin{rem}\label{r:b2}
Observe that $b(G,H) \geqs 3$ if $H$ is a core-free subgroup of $G$ and $|H|^2 \geqs |G|$. This is clear if $|H|^2>|G|$, so let us assume $|H|^2 = |G|$. If $b(G,H) = 2$ then there exists $g \in G$ such that the double coset $HgH$ has size $|H|^2 = |G|$, so $G = HgH$ and thus $G = HH^g$. Then $g = hk$ with $h \in H$ and $k \in H^g$, so $H = H^{gk^{-1}} = H^g$ and we have reached a contradiction.
\end{rem}

We will need the extension of Theorem \ref{t:fkl} for almost simple groups presented in Proposition \ref{p:large} below. In order to state this result, we require the following definition. Recall that a subgroup $H$ of a finite group $G$ is \emph{large} if $H \ne G$ and $|H|^2>|G|$.

\begin{defn}\label{d:beta}
Let $G$ be a finite group and define $\b(G)$ to be the number of conjugacy classes of large core-free subgroups of $G$. Notice that $\a(G) \geqs \b(G)$.
\end{defn}

\begin{prop}\label{p:large}
Let $G$ be a finite almost simple group. Then $\b(G) \geqs 1$, with equality if and only if $G$ is one of the following:
\begin{itemize}\addtolength{\itemsep}{0.2\baselineskip}
\item[{\rm (i)}] $G = {\rm L}_2(q)$ and $(q-1)/(2,q-1) \geqs 7$ is a prime.
\item[{\rm (ii)}] $G = {}^2B_2(q)$ and $q-1$ is a Mersenne prime.
\item[{\rm (iii)}] $G = {\rm L}_{2}(3^5)$, ${\rm J}_1$, ${\rm J}_3$, ${\rm J}_3{:}2$ or 
$\mathbb{M}$.
\end{itemize}
\end{prop}

\begin{proof}
Let $H_1, \ldots, H_k$ be a complete set of representatives of the conjugacy classes of the core-free maximal subgroups of $G$, where $|H_i| \geqs |H_{i+1}|$ for all $i$. Let $T$ be the socle of $G$.

First assume $T$ is a sporadic simple group. If $T = {\rm O'N}$ then $G$ contains subgroups $H = {\rm L}_3(7).2$ and ${\rm L}_3(7)$, both of which are large. In every other case, by inspecting the tables of maximal subgroups in \cite{Wilson} it is routine to check that $|H_1|^2 > |G|$, so $\b(G) \geqs 1$. Similarly, if $T \not\in \{ {\rm J}_1, {\rm J}_3, \mathbb{M}\}$ then we also find that $|H_2|^2 > |G|$ and thus $\b(G) \geqs 2$. If $T = {\rm J}_1$ or ${\rm J}_3$ then $\max\{\frac{1}{4}|H_1|^2, |H_2|^2\}<|G|$ and we conclude that $\b(G) = 1$. Similarly, if $G = \mathbb{M}$ then $|H_2|^2 < |G|$ and we claim that $|J|^2<|G|$ for every maximal subgroup $J$ of $H_1 = 2.\mathbb{B}$, which implies that $\b(G) = 1$. To see this, first note that $J$ contains $Z = Z(H_1) = C_2$, so $J/Z$ is a maximal subgroup of $H_2/Z = \mathbb{B}$. This implies that $|J| \leqs 2|2.{}^2E_6(2){:}2|$ and we deduce that $|J|^2 < |G|$ as claimed.

Next suppose $T = A_n$ is an alternating group. If $n=5$ then it is straightforward to check that $|H_2|^2>|G|$ and thus $\b(G) \geqs 2$. Similarly, if $n \geqs 7$ then $G = S_n$ or $A_n$ and by considering the stabilizers of $1$-sets and $2$-sets (with respect to the natural action on $\{1, \ldots, n\}$) we deduce that $|H_2| \geqs (n-2)!$, which immediately yields $|H_2|^2>|G|$. Finally, if $n=6$ then with the aid of {\sc Magma} \cite{Magma} it is easy to check that $\b(G) \geqs 2$ (in fact, $|H_2|^2 > |G|$ unless $G = {\rm PGL}_2(9)$ or ${\rm M}_{10}$). 

To complete the proof, we may assume $T$ is a simple group of Lie type over $\mathbb{F}_q$, where $q=p^f$ for a prime $p$ and positive integer $f$. If $T$ is a classical group, then let $V$ be the natural module. We divide the analysis into two cases according to the twisted Lie rank of $T$, denoted by ${\rm rk}(T)$.

\vs

\noindent \emph{Case 1. ${\rm rk}(T)=1$.} 

\vs

We begin by assuming $T$ has twisted Lie rank $1$, so $T = {\rm L}_2(q)$, ${\rm U}_3(q)$, ${}^2B_2(q)$ or ${}^2G_2(q)'$. We consider each possibility in turn.

Suppose $T = {\rm L}_2(q)$. Set $d=(2,q-1)$ and $|G:T|=m$, so $m \leqs df$. Since we have already treated the alternating groups, we may assume $q \geqs 7$ and $q \ne 9$. The groups with $q \leqs 19$ can be handled using {\sc Magma}, so we will assume $q \geqs 23$. Then $H_1$ is a Borel subgroup of order $q(q-1)m/d$, so $|H_1|^2>|G|$ and $\b(G) \geqs 1$. If $m>1$ then it is straightforward to check that $|H|^2 > fq(q^2-1) \geqs |G|$, where $H = H_1 \cap T$ has order $q(q-1)/d$. Therefore, we may assume $G=T$.

Suppose $p=2$, so $f \geqs 5$. If $f$ is even then $q-1$ is divisible by $3$ and $|H|^2>|G|$ where $H$ is an index-three subgroup of $H_1$. Now assume $f$ is odd and note that $|H_2|^2<|G|$ (see \cite[Tables 8.1, 8.2]{BHR}, for example). If $q-1$ is a prime then the largest proper subgroup of $H_1$ has order $q$ and we deduce that $\b(G) = 1$. Now assume $q-1$ is composite and let $r$ be the smallest prime divisor of $q-1$. Then the largest proper subgroup of $H_1$ has order $q(q-1)/r$, so $\b(G) \geqs 2$ if and only if
\begin{equation}\label{e:eq1}
r^2 < \frac{q(q-1)}{q+1}.
\end{equation}
By \cite[Lemma 2.6]{BTV}, there are no solutions to the equation $2^f-1 = r^{\ell}$ with $\ell \geqs 2$, so $q-1$ must be divisible by at least two distinct primes and thus $r(r+2) \leqs q-1$. This implies that $r \leqs \sqrt{q}-1$ and one can now check that \eqref{e:eq1} holds.

Now assume $p$ is odd and $q \geqs 23$. If $f$ is even then $H = {\rm PGL}_2(q^{1/2})$ is large, so let us assume $f$ is odd. Then $|H_2|^2<|G|$, so we need to determine if $H_1$ has a proper subgroup $H$ with $|H|^2>|G|$. If $(q-1)/2$ is a prime then the largest proper subgroup of $H_1$ has order $q$ and we deduce that $\b(G) = 1$. Now assume $(q-1)/2$ is composite and let $r$ be the smallest prime divisor of $(q-1)/2$. Then $\b(G) \geqs 2$ if and only if
\begin{equation}\label{e:eq2}
r^2 < \frac{q(q-1)}{2(q+1)}.
\end{equation}
If $(q-1)/2 \ne r^2$ then $r(r+2) \leqs (q-1)/2$, so $r \leqs \sqrt{(q+1)/2}-1$ and \eqref{e:eq2}  holds. Now assume $(q-1)/2=r^2$, so we are interested in solutions to the Diophantine equation
\[
p^f + (-1)^f = 2r^2.
\]
If $f \geqs 5$ then \cite[Theorem 1.1]{BS} implies that $(p,f,r) = (3,5,11)$ is the only solution, whence $\b(G) = 1$ when $q = 3^5$. If $f=3$ then $(p^2+p+1)(p-1) = 2r^2$ and this possibility is quickly eliminated. Similarly, if $f=1$ then $p = 2(r^2-1)+3$ and thus $r=3$ and $p=19$, which is incompatible with our bound $q \geqs 23$. We conclude that $q=3^5$ is the only case that arises here.

Next assume $T = {\rm U}_3(q)$, so $q \geqs 3$. Set $d=(3,q+1)$ and $|G:T|=m$, so $m \leqs 2df$ and we note that $|H_1| = q^3(q^2-1)m/d$ and $|H_1|^2>|G|$ (here $H_1$ is a Borel subgroup). If $m>1$ then it is easy to check that $|H|^2 > 2df|T| \geqs |G|$, where $H = H_1 \cap T$, so we may assume $G = T$. If $d=1$ then $H = {\rm GU}_2(q)$ is a maximal subgroup of $G$ and one checks that $|H|^2>|G|$. Now assume $d=3$, so $q \geqs 5$ and $q^2-1$ is divisible by $3$. This implies that $H_1$ has a subgroup $H$ with $|H_1:H|=3$ and it is easy to check that $|H|^2>|G|$. 

Now suppose $T = {}^2B_2(q)$, so $p=2$, $f \geqs 3$ is odd and $|G:T|=m$ with $m \leqs f$. Here $H_1$ is a Borel subgroup and we have $|H_1| = q^2(q-1)m$ and $|H_1|^2>|G|$. If $m>1$ then take $H = H_1 \cap T$. Now assume $G=T$ and note that $|H_2|^2<|G|$ (see \cite[Table 8.16]{BHR}). If $q-1$ is a prime, then the largest subgroup of $H_1$ has order $q^2$ and thus $\b(G)=1$. Now assume $q-1$ is composite and let $r$ be the smallest prime divisor of $q-1$. Then $H_1$ has a subgroup $H$ of index $r$ and we have $\b(G) \geqs 2$ if and only if
\begin{equation}\label{e:eq3}
r^2 < \frac{q^2(q-1)}{q^2+1}.
\end{equation}

By arguing as above, we know that $q-1$ is divisible by at least two distinct primes. Suppose $f$ is composite and let $s$ be the smallest prime divisor of $f$. Then $s \leqs f/3$ and $2^s-1$ divides $q-1$, so $r \leqs 2^{f/3}-1$ and it is straightforward to check that the bound in \eqref{e:eq3} is satisfied. Now assume $f$ is a prime, so Fermat's Little Theorem implies that every prime divisor of $2^f-1$ is congruent to $1$ modulo $f$. If $r'$ is the second smallest prime divisor of $q-1$ then $r' \geqs r+2f$ and thus $r(r+2f) \leqs q-1$. In turn, this implies that $r \leqs \sqrt{q+f^2-1}-f$ and one can now check that the bound in \eqref{e:eq3} is satisfied.

To complete the proof for rank $1$ groups, let us assume $T = {}^2G_2(q)'$, so $p=3$, $f$ is odd and $|G:T|=m$ with $m \leqs f$. If $f = 1$ then $T \cong {\rm L}_2(8)$, so we may assume $f \geqs 3$. If we take $H$ to be an index-two subgroup of a Borel subgroup of $T$, then $|H| = q^3(q-1)/2$ and $|H|^2 >f|T| \geqs |G|$, so $\b(G) \geqs 2$.

\vs

\noindent \emph{Case 2. ${\rm rk}(T) \geqs 2$.}

\vs

Suppose $T = {\rm L}_n(q)$ with $n \geqs 3$ and set $d = (n,q-1)$. First we claim that $|H|^2>|G|$ when $H$ is a $P_1$ parabolic subgroup of $T$ (that is, $H$ is the stabilizer in $T$ of a $1$-dimensional subspace of $V$). Since $|G:T| \leqs 2df$, $|T:H| = (q^n-1)/(q-1) < q^n$ and 
\[
|T| = \frac{1}{d}q^{n(n-1)/2}\prod_{i=2}^n(q^i-1) > \frac{1}{2d}q^{n^2-1},
\]
we deduce that $|H|^2>|G|$ if $q^{n^2-n-1} > 4d^2f$. One can check that the latter bound holds unless $n=3$ and $q \in \{2,4\}$ and both of these cases can be handled directly since we have $|H| = q^3(q-1)(q^2-1)/d$. This shows that $\b(G) \geqs 1$. Similarly, for $n \geqs 4$ it is very easy to show that $|H|^2>|G|$ when $H$ is a $P_2$ maximal parabolic subgroup of $T$. For $n=3$ we take $H = P_{1,2}$ to be a Borel subgroup of $G$, so $|G:H| = (q+1)(q^2+q+1)$ and we deduce that $|H|^2>|G|$ unless $q=2$. Finally, if $(n,q) = (3,2)$ then we can take the subgroup $H = 7{:}3$ of $T$.

Next assume $T = {\rm U}_n(q)$ with $n \geqs 4$. Set $d = (n,q+1)$ and let $H = P_2$ be the stabilizer in $G$ of a totally isotropic $2$-dimensional subspace of $V$. Then 
\[
|G:H| < 2\left(\frac{q+1}{q}\right)q^{4n-12}
\]
and it is straightforward to check that $|H|^2>|G|$ since $|G|>\frac{1}{2d}q^{n^2-1}$. If $K$ is a $P_1$ parabolic subgroup of $G$, then $|K|>|H|$ and we conclude that $\b(G) \geqs 2$.

Now suppose $T = {\rm PSp}_n(q)'$ with $n \geqs 4$ even and set $d=(2,q-1)$. If $n \geqs 6$ and $H = P_2$ is the stabilizer in $G$ of a totally isotropic $2$-space, then
\[
|G:H| < 2\left(\frac{q}{q-1}\right)q^{2n-5}
\]
and the bound $|G|>\frac{1}{2d}q^{n(n+1)/2}$ implies that $|H|^2>|G|$. As in the previous case, we have $|K|>|H|$ when $K = P_1$, whence $\b(G) \geqs 2$. Now assume $n=4$. Since ${\rm PSp}_4(2)' \cong A_6$ and ${\rm PSp}_4(3) \cong {\rm U}_4(2)$, we may assume $q \geqs 4$. If $H$ is a Borel subgroup of $G$ then $|G:H| = (q^2+1)(q+1)^2$ and it is easy to check that $|H|^2 > |G|$, which implies that $\b(G) \geqs 2$. 

Similarly, if $T = {\rm P\O}_n^{\e}(q)$ is an orthogonal group with $n \geqs 7$ then it is a routine exercise to check that $|H|^2>|G|$ when $H$ is either a $P_1$ parabolic subgroup of $T$, or the stabilizer in $T$ of a nonsingular $1$-space. We omit the details.

To complete the proof, we may assume $T$ is an exceptional group of Lie type. If $G = G_2(q)'$ then we may assume $q \geqs 3$ (since $G_2(2)' \cong {\rm U}_3(3)$) and one can check that $|H|^2>|G|$ if $H$ is a Borel subgroup of $G$ (note that $|H \cap T| = q^6(q-1)^2$). For each of the remaining exceptional groups with ${\rm rk}(T) \geqs 2$, it is a straightforward exercise to exhibit two maximal parabolic subgroups of distinct orders with the desired property. The result follows. 
\end{proof}

As a corollary of Proposition \ref{p:large} (and its proof), we can determine all the special simple groups. This establishes Theorem \ref{t:special}.

\begin{thm}\label{t:spec}
Let $T$ be a nonabelian finite simple group. Then $T$ is special if and only if $\b(T)=1$.
\end{thm}

\begin{proof}
If $T = {\rm Aut}(T)$, then $T$ is special if and only if $\b(T) = 1$. More generally, if $\b(T) = 1$ then $T$ is special, so we may assume $T \ne {\rm Aut}(T)$ and $\b(T) \geqs 2$. We now inspect the proof of Proposition \ref{p:large}, noting that every group with $\b(T) \geqs 2$ contains two large subgroups $H$ and $K$ with $H \not\cong K$. Clearly, $H$ and $K$ are not ${\rm Aut}(T)$-conjugate and thus $T$ is not special. The result follows.
\end{proof}

Finally, we establish Theorem \ref{t:asb2}.

\begin{proof}[Proof of Theorem \ref{t:asb2}]
Let $G$ be a finite almost simple group with socle $T$. In view of the bound $\a(G) \geqs \b(G)$, we may assume $G$ is one of the groups in Proposition \ref{p:large} with $\b(G)=1$. Let $H$ be a core-free subgroup of $G$. 

First consider the sporadic examples listed in part (iii) of Proposition \ref{p:large}. If $G = {\rm L}_2(3^5)$ then $\a(G) \geqs 2$ since $b(G,H) = 3$ when $H=3^5{:}11$ is a proper subgroup of a Borel subgroup of $G$. Similarly, if $G = {\rm J}_3{:}c$ with $c \in \{1,2\}$ then $b(G,H) = 3$ for $H = {\rm L}_{2}(16).2c$ and ${\rm L}_{2}(16).c$. Next assume $G = {\rm J}_1$. If $H$ is maximal, then the main theorem of \cite{BOW} states that $b(G,H) \geqs 3$ if and only if $H = {\rm L}_2(11)$ and it is straightforward to check (with the aid of {\sc Magma}, for example) that $b(G,K) = 2$ for every proper subgroup $K$ of ${\rm L}_{2}(11)$. Therefore, $\a(G)=1$. Similarly, $\a(G) = 1$ for $G = \mathbb{M}$ by \cite[Theorem 3.1]{B_spor}.

Next assume $G = {}^2B_2(q)$ and $q-1$ is a Mersenne prime. Let $B = U{:}T$ be a Borel subgroup of $G$ with unipotent radical $U = [q^2]$ and torus $T = C_{q-1}$. Since $|B|^2>|G|$ we have $b(G,B) \geqs 3$. If $H$ is not contained in a conjugate of $B$, then \cite[Proposition 4.40]{BLS} implies that $b(G,H) = 2$, whence $\a(G) = 1$ if and only if $b(G,H) \leqs 2$ for every proper subgroup $H$ of $B$. Since $q-1$ is a prime, every proper subgroup of $B$ is contained in $U$ or $V{:}T$, up to conjugacy, where $V$ is the unique normal subgroup of $U$ of order $q$. We claim that $b(G,H) = 2$ for $H = U$ or $V{:}T$,  which implies that $\a(G) = 1$. 

For $H = U$ this is clear since $U \cap U^g = 1$ for all $g \in G \setminus B$. Now assume $H = V{:}T$.  Let $B^g$ be another Borel
subgroup with $U \cap  U^g =1$.  By the Bruhat decomposition, we may assume that $g \in N_G(T)$ and so $B^g \cap B =T$. 
Let $u \in U$ be an element of order $4$.  Then $H^u = V{:}T^u$ and
\[
H^g \cap H^u \leqs B^g \cap B = T.
\]
Therefore, $H^g \cap H^u \leqs T \cap H^u$ and we observe that $T \cap H^u = 1$ since $T$ acts fixed point freely on $U/V$, which yields $b(G,H) = 2$. We conclude that $\a(G) = 1$ if $G = {}^2B_2(q)$ and $q-1$ is a prime.

Finally, let us assume $G = {\rm L}_2(q)$, where $(q-1)/(2,q-1) \geqs 7$ is a prime. If $q$ is even then $b(G,H) = 3$ when $H$ is a Borel subgroup and the same conclusion holds for  $H = D_{2(q+1)}$ by \cite[Lemma 4.8]{B20}. Therefore, $\a(G) \geqs 2$. Now assume $q \geqs 23$ is odd and write $q=p^f$ with $p$ a prime. Note that $q \equiv 3 \imod{4}$, so $f$ is odd. Let $B = U{:}T$ be a Borel subgroup, where $U = (C_p)^f$ is the unipotent radical and $T = C_{(q-1)/2}$ is a torus. Here $|B|^2 > |G|$ and thus $b(G,B) \geqs 3$. By combining results from \cite{B20} and \cite{BH}, we observe that $b(G,H) = 2$ if $H$ is any other maximal subgroup of $G$. More precisely, we apply \cite[Lemma 4.7]{B20} if $H = D_{q-1}$ and \cite[Lemma 4.8]{B20} for $H = D_{q+1}$. The subgroups of type $2^{1+2}_{-}.{\rm O}_{2}^{-}(2)$ are handled in \cite[Lemma 4.10]{B20}, while the subfield subgroups (noting that $f$ is odd) and $A_5$ subgroups are treated in the proof of \cite[Theorem 4.22]{BH}. Therefore, $\a(G) = 1$ if and only if $b(G,H) \leqs 2$ for every proper subgroup $H$ of $B$. 

Let $H$ be a proper subgroup of $B$. Then $H$ is contained in either $U$ or $T$, up to conjugacy, so it suffices to show that $b(G,U) = b(G,T) = 2$. Since $U \cap U^g = 1$ for all $g \in G \setminus B$ we have $b(G,U) = 2$. In addition, we have already noted that $b(G,N) =2$ for $N = N_G(T) = D_{q-1}$, whence $b(G,T) = 2$ and we conclude that $\a(G) = 1$.
\end{proof}

\section{Some general lemmas}\label{s:general}

For the remainder of the paper, we will assume $G$ is a finite group with trivial Frattini subgroup and generalized Fitting subgroup $F^*(G) = E(G)F(G)$, where $E(G)$ and $F(G)$ denote the layer of $G$ and the Fitting subgroup of $G$, respectively. Recall that the socle of $G$, denoted ${\rm soc}(G)$, is the subgroup generated by the minimal normal subgroups of $G$.

\begin{lem}\label{l:basic}
We have ${\rm soc}(G) = F^*(G) = E(G) \times F(G)$ and $E(G)$ is a direct product of nonabelian simple groups.
\end{lem}

\begin{proof}
This is straightforward (see \cite[Lemma 2.2]{AB}, for example).
\end{proof}

We will repeatedly apply the following elementary observation.

\begin{lem}\label{l:basic2}
Let $N$ be a normal subgroup of $G$ and let $H$ be a proper subgroup of $N$ with $|H|^2 > |N|$. If $H$ is core-free, then $b(G,H) \geqs 3$. In particular, $b(G,H) \geqs 3$ if $N$ is a minimal normal subgroup of $G$.
\end{lem}

\begin{proof}
If $g \in G$ then $H \cap H^g \leqs N$ and so the bound $|H|^2 > |N|$ implies that $H \cap H^g \ne 1$. Therefore, $b(G,H) \geqs 3$ if $H$ is core-free.
\end{proof}

We now present a sequence of lemmas with the aim of obtaining strong restrictions on the structure of the socle of a strongly base-two group. In the proofs, we will repeatedly apply Lemma \ref{l:basic2}.

\begin{lem}\label{l:1}
Suppose $E(G) \ne 1$. Then $\a(G) \geqs 1$, with equality only if 
\[
F^*(G) = T \times C_{p_1} \times \cdots \times C_{p_k},
\]
where $T$ is a special simple group, $k \geqs 0$ and each $C_{p_i}$ is a minimal normal subgroup of $G$.  
\end{lem}

\begin{proof}
Let $N_1 = T^m$ be a minimal normal subgroup of $G$, where $T$ is a nonabelian simple group and $m \geqs 1$. In view of Proposition \ref{p:large}, let $J$ be a large subgroup of $T$. 

Suppose $m \geqs 2$ and consider the proper subgroups $H_1 = T^{m-1} \times J$ and $H_2 = T^{m-2} \times J^2$ of $N_1$. Then
\[
|H_i|^2 \geqs |T|^{2m-4}|J|^4 > |T|^{2m-2} \geqs |N_1|
\]
and thus $b(G,H_i) \geqs 3$ by Lemma \ref{l:basic2}, so $\a(G) \geqs 2$ and for the remainder we may assume $m=1$. Next suppose $N_2=S$ is another minimal normal subgroup of $G$, where $S$ is a nonabelian simple group. Fix a large subgroup $K$ of $S$ and set $H_1 = J$ and $H_2 = K$. Then the $H_i$ are core-free and non-conjugate, with $b(G,H_i) \geqs 3$ by Lemma \ref{l:basic2}, so $\a(G) \geqs 2$.

We have now reduced to the case where $E(G) = N_1 = T$ is simple. Note that $b(G,J) \geqs 3$ by Lemma \ref{l:basic2}, so $\a(G) \geqs 1$. Suppose $T$ 
is non-special and choose a large subgroup $K$ of $T$ such that $J,K$ are not ${\rm Aut}(T)$-conjugate. Then $b(G,K) \geqs 3$ and $J,K$ are non-conjugate in $G$ since $T$ is normal in $G$, whence $\a(G) \geqs 2$. Finally, let us assume $T$ is special and $G$ has a minimal normal subgroup $N_2 =(C_p)^{a}$ with $a \geqs 2$. If we set $H_1 = J$ and $H_2 = J \times (C_p)^{a-1}$, then $|H_1|^2 > |N_1|$ and $|H_2|^2 > |N_1||N_2|$, so the $H_i$ are core-free subgroups with $b(G,H_i) \geqs 3$ and thus $\a(G) \geqs 2$.
\end{proof}

\begin{lem}\label{l:3}
Suppose $G$ has a minimal normal subgroup $N = (C_p)^a$ with $a \geqs 5$. Then $\a(G) \geqs 2$. 
\end{lem}

\begin{proof}
We can take subgroups $H_1 = (C_p)^{a-1}$ and $H_2 = (C_p)^{a-2}$ in $N$, noting that $|H_i|^2 > |N|$ for $i=1,2$. 
\end{proof}

\begin{lem}\label{l:4}
Suppose $G$ has minimal normal subgroups $N_1 = (C_p)^a$ and $N_2 = (C_q)^b$, where $a \geqs 3$ and $b \geqs 2$. Then $\a(G) \geqs 2$. 
\end{lem}

\begin{proof}
Set $H_1 = (C_p)^{a-1}$ and $H_2 = (C_p)^{a-1} \times (C_q)^{b-1}$, with respective normal closures $N_1$ and $N_1 \times N_2$, respectively. Then $|H_1|^2 > |N_1|$, $|H_2|^2 > |N_1||N_2|$ and each $H_i$ is core-free. Therefore, $b(G,H_i) \geqs 3$ and thus $\a(G) \geqs 2$ as required.
\end{proof}

\begin{lem}\label{l:5}
Suppose $G$ has minimal normal subgroups $N_1 = (C_p)^a$ and $N_2 = C_p$, where $a \geqs 3$. Then $\a(G) \geqs 2$. 
\end{lem}

\begin{proof}
Let $H_1$ and $H_2$ be subgroups of $N_1 \times N_2$ of order $p^{a-1}$ and $p^a$, respectively, where $H_1<N_1$ and $H_2$ has normal closure $N_1 \times N_2$. Then the $H_i$ are core-free and it is clear that $b(G,H_i) \geqs 3$, so $\a(G) \geqs 2$ as required.
\end{proof}

\begin{lem} \label{l:6}  
Suppose $G$ has minimal normal subgroups $N_i = (C_p)^{a_i}$ for $i=1,2,3$, where $a_1 \geqs 2$. Then $\a(G) \geqs 2$.
\end{lem}

\begin{proof}  
There exist core-free subgroups $H_1 < N_1 \times N_2$ and $H_2 < N_1 \times N_3$ of index $p$, whence $b(G,H_i) \geqs 3$ and the result follows.
\end{proof}

Note that if $G$ has an abelian minimal normal subgroup $N = (C_p)^a$, then we may view $N$ as an irreducible $a$-dimensional $G$-module over the field $\mathbb{F}_p$.

\begin{lem} \label{l:7}  
Suppose $G$ has minimal normal subgroups $N_i = (C_p)^{a_i}$ for $i=1,2,3$, which are pairwise non-isomorphic as $G$-modules. Then $\a(G) \geqs 1$, with equality only if $|O_p(G)| = p^3$ and $a_i=1$ for all $i$.
\end{lem}

\begin{proof}  
If $a_i \geqs 2$ for some $i$, then apply Lemma \ref{l:6}, so we may assume $a_i = 1$ for all $i$. Note that $N_1 \times N_2 \times N_3$ has a core-free subgroup $H$ of index $p$, which implies that $\a(G) \geqs 1$. If $|O_p(G)| \geqs p^4$, then $G$ has a minimal normal $p$-subgroup $M$ such that $M \ne N_1$ and $M$ is not isomorphic to $N_2$ or $N_3$  as $G$-modules.  Then $M \times N_2 \times N_3$ contains a core-free subgroup of index $p$ whose normal closure is not equal to the normal closure of $H$. Therefore, $\a(G) \geqs 2$ and the result follows.
\end{proof}

Let $N$ be an abelian normal subgroup of $G$. Since $N \cap \Phi(G) = 1$, it follows that $N$ is complemented in $G$ (see \cite[Lemma 2.1(vii)]{AB}, which is attributed to Gasch\"{u}tz \cite{Gas}), which means that $G = N{:}A$ for some subgroup $A$ of $G$. Since $C_G(F^*(G)) = Z(F^*(G))$ (see \cite[(31.13)]{Asch}, for example) and Lemma \ref{l:basic} gives $Z(F^*(G)) = F(G)$, we deduce that  
\begin{equation}\label{e:FGL}
G = F(G){:}L,
\end{equation}
where $L$ is isomorphic to a subgroup of ${\rm Aut}(F^*(G))$.  

\begin{lem} \label{l:8}   
Suppose that $H$ is a core-free $p$-subgroup of $F(G)$ such that $|H|^2 > |N|$, 
where $N$ is the normal closure of $H$. Then $\a(G) \geqs 1$, with equality only if the following conditions are satisfied: 
\begin{itemize}\addtolength{\itemsep}{0.2\baselineskip}
\item[{\rm (i)}]  $E(G)=1$. 
\item[{\rm (ii)}] $N_L(H)=1$ and $L$ acts faithfully on $N$.
\item[{\rm (iii)}] $F(G) = N \times Q$, where $Q = O_{p'}(G)$ and every subgroup of $Q$ is normal in $G$.
\end{itemize} 
\end{lem}

\begin{proof} 
First observe that $b(G,K) \geqs 3$ for every core-free subgroup $K$ of $G$ containing $H$. In particular, $b(G,H) \geqs 3$ and thus $\a(G) \geqs 1$. 
For the remainder, let us assume $\a(G) = 1$. If $E(G)$ is nontrivial then it contains a core-free subgroup $J$ with $|J|^2 > |E(G)|$ (see the proof of Lemma \ref{l:1}), so $b(G,J) \geqs 3$ and $\a(G) \geqs 2$. Therefore (i) holds. 

Now $H$ is a maximal core-free subgroup of $G$ and the conditions in part (ii) hold since $F(G) = {\rm soc}(G)$ and  $\langle H, N_L(H) \rangle \cap F(G) = H$, whence $HN_L(H)$ is core-free.  

Next observe that $N \leqs O_p(G)$. Write $F(G) = N \times Q$ with
$Q$ normal in $G$. The maximality of $H$ among core-free subgroups of $G$ implies that every subgroup of $Q$ is normal and so it remains to show that $Q = O_{p'}(G)$. The inequality $|H|^2 > |N|$ implies that $N$ contains a minimal normal subgroup $B$ with $|B| \geqs p^2$. If $Q \ne O_{p'}(G)$, then $Q$ contains a normal subgroup $C$ of order $p$ and thus $B \times C$ contains a core-free subgroup $D$ with $|D|=|B|$ and normal closure $B \times C \ne N$. Therefore,  $b(G,D) \geqs 3$, which is incompatible with our hypothesis $\a(G) = 1$. Therefore $Q=O_{p'}(G)$ and (iii) holds. 
\end{proof} 

We close this section with two elementary observations in the case $E(G)=1$. 

\begin{lem} \label{l:10}  
Suppose $E(G)=1$ and $A$ is an abelian subgroup such that $A \cap F(G) =1$.
\begin{itemize}\addtolength{\itemsep}{0.2\baselineskip}
\item[{\rm (i)}]  If $A$ acts semisimply on $F(G)$, then $b(G,A) \leqs 2$.
\item[{\rm (ii)}] If $b(G,A) \geqs 3$, then $\a(G) \geqs 2$.
\end{itemize}
\end{lem}

\begin{proof}
Suppose $A$ acts semisimply on $F(G)$ and write $F(G) = N_1 \times \cdots \times N_k$, where each $N_i$ is a minimal normal subgroup of $F(G){:}A$. Fix $x \in F(G)$ so that the projection of $x$ onto each $N_i$ is nontrivial and let $J$ be the $A$-submodule of $F(G)$ generated by $x$.  Since $x$ projects onto each minimal normal
subgroup of $F(G)$, it follows that each $N_i$ is isomorphic to an $A$-submodule of $J$ and so $A$ acts faithfully on $J$. Therefore, since $C_A(x)$ acts trivially
on $J$, we deduce that $C_A(x) =1$. Finally, suppose that $a \in A \cap A^x$.  Then $a^x \in aF(G) \cap A$, so $a^x = a \in C_A(x)=1$ and we conclude that  $b(G,A) \leqs 2$.   

Now suppose $b(G,A) \geqs 3$. As in \eqref{e:FGL}, write $G = F(G){:}L$ and let $A_L$ be the subgroup of $L$ such that $F(G){:}A_L=F(G){:}A$. Then the orbits of $A$ and $A_L$ on $F(G)$ are the same and thus $b(G,A_L) \geqs 3$. By (i), $A_L$ does not act semisimply on $F(G)$ and so $A_L$ is a proper subgroup of $L$. Since $b(G,L) \geqs 3$, we conclude that $\a(G) \geqs 2$.
\end{proof} 

\begin{cor} \label{c:9}  
If $E(G)=1$ and every subgroup of $F(G)$ is normal in $G$, then $\a(G)=0$.
\end{cor}

\begin{proof}  
Let $H$ be any core-free subgroup of $G$. Then $H \cap F(G)=1$, so $H$ acts faithfully and semisimply on $F(G)$ and it follows that $H$ is abelian. Now apply Lemma \ref{l:10}(i). 
\end{proof} 

\section{$E(G) \ne 1$}\label{s:t3}

In this section, we prove Theorem \ref{t:main1}, which describes the finite strongly base-two groups $G$ with $\Phi(G)=1$ and $E(G) \ne 1$. 

\begin{proof}[Proof of Theorem \ref{t:main1}]
Let $G$ be a finite group such that $\Phi(G)=1$ and $E(G) \ne 1$. By Lemma \ref{l:1} we have $\a(G) \geqs 1$, with equality only if 
\[
F^*(G) = T \times C_{p_1} \times \cdots \times C_{p_k},
\]
where $T$ is a special simple group, $k \geqs 0$ and each $C_{p_i}$ is a minimal normal subgroup of $G$. Let $H$ be a representative of the unique ${\rm Aut}(T)$-class of large subgroups of $T$. For now, let us assume $\a(G) = 1$.  By inspecting the possibilities for $T$ and $H$ (see Table \ref{tab:special}), we observe that $T$ has a unique $T$-class of large subgroups. In turn, this implies that if $T < A \leqs {\rm Aut}(T)$, then $A =N_A(H)T$ and thus $N_A(H) \ne H$. 

If $F(G) = 1$ then $G$ is almost simple and the result follows from Theorem \ref{t:asb2}. Now assume $F(G) \ne 1$ and write 
\[
F^*(G) = T \times (C_{p_1})^{a_1} \times \cdots \times (C_{p_t})^{a_t},
\]
where the $p_i$ are distinct primes. Recall that $G = F(G){:}L$ (see \eqref{e:FGL}) with 
\[
T \normeq L \leqs {\rm Aut}(F^*(G)) = {\rm Aut}(T) \times {\rm GL}_{a_1}(p_1) \times \cdots \times {\rm GL}_{a_t}(p_t).
\]
If $x \in C_L(T)$ is nontrivial then $J = H \times \la x \ra$ is core-free and $b(G,J) \geqs 3$. Therefore, we have $C_L(T) = 1$ and thus the projection map $\pi:L \to {\rm Aut}(T)$ is injective, which implies that $L$ is an almost simple group with socle $T$. Set $K = N_L(H)$. 

As noted above, if $L \ne T$ then $H<K$ and $K$ is core-free in $G$, so $b(G,K) \geqs 3$ since $b(G,H) \geqs 3$. Therefore, the condition $\a(G)=1$ implies that $L=T$ and thus $G = T \times F(G) = T \times Z(G)$. Finally, suppose $\a(T) \geqs 2$, say $H,K<T$ are non-conjugate subgroups such that $b(T,H) \geqs 3$ and $b(T,K) \geqs 3$. Then $H$ and $K$ are non-conjugate core-free subgroups of $G$ and we deduce that $\a(G) \geqs 2$ since $b(G,H) = b(T,H)$ and similarly for $K$. Therefore, $\a(T)=1$ and the result follows. 

To complete the proof, let us assume $G = T \times Z(G)$ and $\a(T)=1$. We need to show that $\a(G)=1$. To see this, first observe that every core-free subgroup of $G$ is of the form 
\[
S_f = \{(s,f(s))\,:\, s \in S\},
\]
where $S<T$ is a proper subgroup and $f:S \to Z(G)$ is a homomorphism. Note that if $b(T,S) = 2$ then $b(G,S_f) = 2$ for every homomorphism $f$, so it suffices to show that $b(G,H_f) = 2$ when $f$ is nontrivial. Let us also note that
\[
H_f \cap (H_f)^t = \{(h,f(h^t)) \,:\, h \in H \cap H^t, \, f(h) = f(h^t)\}
\]
for all $t \in T$. We now consider each possibility for $T$ and $H$ (see Theorem \ref{t:asb2} and Table \ref{tab:special}). 

First assume $T = {\rm L}_2(q)$ and $H = [q]{:}(q-1)/2$, where $q \geqs 23$ is odd and $(q-1)/2$ is a prime. Let $f:H \to Z(G)$ be a nontrivial homomorphism, which implies that $\ker(f) = [q]$ and ${\rm im}(f) = C_{(q-1)/2}$. Fix $t \in T \setminus H$ and note that $H \cap H^t = \la h \ra = C_{(q-1)/2}$. Now $H$ has $(q-3)/2$ distinct classes of elements of order $(q-1)/2$, while there are $(q-3)/4$ such classes in $T$. Therefore, we can choose $t$ so that $h$ and $h^t$ are not $H$-conjugate. This implies that $f(h) \ne f(h^t)$ and we conclude that $H_f \cap (H_f)^t = 1$. 

A very similar argument applies when $T = {}^2B_2(q)$ and $H = [q^2]{:}(q-1)$ with $q-1$ a prime. Here $\ker(f) = [q^2]$, ${\rm im}(f) = C_{q-1}$ and $H \cap H^t = C_{q-1}$ for all $t \in T \setminus H$. We can now repeat the previous argument since $H$ has $q-2$ classes of elements of order $q-1$, and there are $(q-2)/2$ such classes in $T$. Finally, if $(T,H) = ({\rm J}_1,{\rm L}_2(11))$ or $(\mathbb{M},2.\mathbb{B})$ then $H/N$ is nonabelian for every proper normal subgroup $N$ of $H$ and thus every homomorphism $f:H \to Z(G)$ is trivial. This completes the proof of Theorem \ref{t:main1}.
\end{proof}

\section{A special case with $E(G)=1$}\label{s:special2} 

For the remainder of the paper, we will assume $G$ is a finite group with $\Phi(G) = E(G) = 1$. In this section, we handle a special case that will be useful in our general analysis. Our main result is Theorem \ref{t:evenderived}. In the statement, we define $r(G)$ to be the maximal rank of a minimal normal subgroup of $G$ (that is, $r(G)$ is the maximal dimension of an irreducible $G$-module contained in $F(G)$).

\begin{thm} \label{t:evenderived} 
Let $G=F(G){:}L$ be a finite group with $\Phi(G) = E(G)=1$. Assume that $r(G) =2$ and the derived subgroup of $L$ has even order. Then $\alpha(G) \geqs 1$, with equality if and only if all of the following conditions are satisfied:
\begin{itemize}\addtolength{\itemsep}{0.2\baselineskip}
\item[{\rm (i)}] $L$ has a unique involution $x$.
\item[{\rm (ii)}] $P = [x,F(G)]= (C_p)^2 = O_p(G)$ is a minimal normal subgroup of $G$ and $p \equiv 3 \imod{4}$.
\item [{\rm (iii)}] $F(G)= P \times Q$, where $Q = O_{p'}(G)$ has odd order and every subgroup of $Q$ is normal in $G$. 
\item[{\rm (iv)}] $|L| =  2(p+1)$ and one of the following holds:

\vspace{1mm}

\begin{itemize}\addtolength{\itemsep}{0.2\baselineskip}
\item[{\rm (a)}] $L = 2.D_{p+1}$ is a double cover of the dihedral group of order $p+1$.
\item[{\rm (b)}]  $L = 2.S_4$ is a double cover of $S_4$ and $p=23$. 
\item[{\rm (c)}] $L={\rm SL}_2(3)$ and $p = 11$.
\item[{\rm (d)}] $L={\rm SL}_2(5)$ and $p = 59$.
\end{itemize}
\end{itemize}
\end{thm}

\begin{proof}
Let $x \in L'$ be an involution and observe that $x$ acts as a scalar on every minimal normal subgroup of $G$. Moreover, $x$ centralizes every normal subgroup of prime order, so $x$ centralizes $O_2(G)$ and $x \in Z(L)$. 

Let $P =[x,F(G)] = [x,P]$ and let $D<P$ be a subgroup such that $|D|^2=|P|$ and $D$ is core-free in $G$. Set $H = \langle D,x \rangle$. Then $x \in N_G(D)$ and $\langle P, x \rangle$ is the normal closure of $H$, whence $b(G,H) \geqs 3$ and thus $\a(G) \geqs 1$. Now  assume $\a(G) = 1$. 

We may write $F(G)=P \times  C_{F(G)}(x)$. If some subgroup $J$ of $C_{F(G)}(x)$ is core-free in $G$, then $HJ$ is core-free and $b(G,HJ) \geqs 3$, so the condition $\a(G)=1$ implies that every subgroup of $C_{F(G)}(x)$ is normal in $G$. Next we claim that $x$ is the unique involution in $L$ as in (i). If $y \in L$ is another involution, then $y$ will normalize some diagonal subgroup $D$ of $P$ as above and by considering the subgroup $\langle D, x, y \rangle$ we deduce that $\a(G) \geqs 2$, a contradiction. This justifies the claim. It  follows that a Sylow $2$-subgroup $S$ of $L$ is generalized quaternion and $S$ acts faithfully on every rank two minimal normal subgroup of $G$.

Let $R=O_2(G) \leqs C_{F(G)}(x)$. Since every subgroup of $R$ is normal in $G$, it follows that $R \leqs Z(G)$. If $R \ne 1$, then we can replace $x$ by $xy$ for some $1 \ne y \in R$ and by arguing as above we see that $b(G,J) \geqs 3$ with $J = \langle D, xy \rangle$, which implies that $\a(G) \geqs 2$ since $x$ and $xy$ are not
conjugate. Therefore, $R=1$.

Write $P = P_1 \times \cdots \times P_m$, where each $P_i$ is a minimal normal subgroup of $G$ with $|P_i|=p_i^2$. The above analysis shows that $L$ acts transitively on the set $X$ of all core-free subgroups $D$ of $P$ with $|D|^2 = |P|$ (indeed, we showed above that each subgroup $D$ of this form gives rise to a core-free subgroup $H = \langle D, x \rangle$ with $b(G,H) \geqs 3$). Since $N_L(D) = \langle x \rangle$ for each $D \in X$, this implies that $L/\langle x \rangle$ acts regularly on $X$, so $|L|=2|X| = 2\prod_i (p_i+1)$ and $|S|=2 \prod_i (p_i+1)_2$.   Since $L$ has a regular orbit on $P$, it follows that $L$ acts faithfully on $P$. In addition, the transitivity of $L$ on $X$ implies that $L$ acts primitively on each $P_i$ (indeed, if $P_i$ is imprimitive, then $L$ has at least two orbits on the lines of $P_i$, which contradicts the fact that $L$ is transitive on $X$). 

We claim that $m=1$ and $p_1 \equiv 3 \imod{4}$. To see this, first assume $p_i \equiv 1 \imod{4}$ for all $i$. Then an element of order $4$ in $S$ fixes a line in each $P_i$, and hence an element of $X$, contradicting the fact that $\langle x \rangle$ is the stabilizer in $L$ of every element in $X$. Therefore, we may assume 
$p_i \equiv 3 \imod{4}$ for some $i$. Then every orbit of $S\langle x \rangle$ on the set of lines of $P_i$ has size $|S|/2 = (p_i+1)_2$ and so $S$ has an orbit on $X$ of size $|S|/2$. This implies that $m=1$ and $p_1 \equiv 3 \imod{4}$, which proves the  claim. 

Set $p = p_1$. By inspecting the subgroups of ${\rm GL}_2(p)$ of order $2(p+1)$ that act transitively on lines (see \cite[6.25, Chapter 3]{Suz}) with a generalized quaternion Sylow $2$-subgroup, we conclude that the possibilities for $L$ are as described in part (iv).

 Finally, we show that $\a(G)=1$ when $G$ is a group satisfying all of the conditions in the theorem. First observe that if $H<F(G)$ is nontrivial and core-free in $G$, then $|H|=p$ and 
thus $b(G,H) =2$. Next, let us note that every nontrivial orbit of $L$ on $P$ is regular (since the stabilizer of a line is $\la x \ra$, and $x$ fixes no nontrivial elements of $P$), whence $b(G,H) = 2$ for every nontrivial core-free subgroup $H$ with $H \cap F(G)=1$. It follows that if $H$ is a core-free subgroup of $G$ and $b(G,H) \geqs 3$, then $H \leqs N_G(D)$, where $D$ is a subgroup of $P$ of order $p$. Now every  maximal core-free subgroup of $N_G(D)$ is conjugate to $\langle D, x \rangle$ and thus $\a(G)=1$ since all subgroups $D$ of this form are conjugate in $G$.   
\end{proof}

\section{$E(G)=1$, $c(G) \ne 2$}\label{s:c4}

Let $p$ be a prime divisor of $|G|$ and set $P = O_p(G)$. Recall that $c_p(G)$ denotes the maximum rank of a cyclic $G$-submodule of $P$ and we define $c(G) \geqs 1$ to be the maximum value of $c_p(G)$ over all primes $p$. Here we handle the groups with $c(G) \ne 2$ and we will establish Theorem \ref{t:EG}. Throughout this section, we write $G = F(G){:}L$ (see \eqref{e:FGL}).

\begin{thm}\label{t:c5}
Suppose $E(G)=1$ and $c(G) \geqs 5$. Then $\a(G) \geqs 2$.
\end{thm}

\begin{proof}
Let $M = (C_p)^t$ be a cyclic $G$-submodule of $O_p(G)$, where $t = c(G) \geqs 5$. Write $M = M_1 \times \cdots \times M_k$, where each $M_i = (C_p)^{a_i}$ is a minimal normal subgroup of $G$. Seeking a contradiction, suppose $\a(G) \leqs 1$.

By Lemma \ref{l:3} we have $k \geqs 2$, and by combining Lemmas \ref{l:4}, \ref{l:5} and \ref{l:6}, we deduce that $a_i = 1$ for all $i$. The result now follows from Lemma \ref{l:7}. 
\end{proof}

Next assume $c(G)=4$. Our main result is the following.  

\begin{thm}\label{t:c4}
Suppose $E(G)=1$ and $c(G) = 4$. Then $\a(G) \geqs 1$, with equality if and only if 
all of the following conditions are satisfied:
\begin{itemize}\addtolength{\itemsep}{0.2\baselineskip}
\item[{\rm (i)}] $F(G) = P \times Q$, where $|P|=16$ and every subgroup of $Q = O_{2'}(G)$ is normal in $G$.
\item[{\rm (ii)}] $L$ acts regularly on the set of core-free subgroups of $P$ of order $8$.
\item[{\rm (iii)}] One of the following holds:

\vspace{1mm}

\begin{itemize}\addtolength{\itemsep}{0.2\baselineskip}
\item[{\rm (a)}] $P = (C_2)^4$ is a minimal normal subgroup of $G$, $L = \la x \ra = C_{15}$ and $\la x^5 \ra$ acts faithfully on $Q$. In particular, $|Q|$ is divisible by a prime $r$ with $r \equiv 1 \imod{3}$. 
\item[{\rm (b)}] $P = A \times B$ and $L = (C_3)^2$, where $A = (C_2)^2$ and $B = (C_2)^2$ are minimal normal subgroups of $G$ that are non-isomorphic as $G$-modules and $C_L(Q)$ is contained in $C_L(A)$ or $C_L(B)$. 
\end{itemize} 
\end{itemize}
\end{thm}

\begin{proof}
Write $c_p(G)=4$ for some prime $p$ and set $P = O_p(G)$. Let $V$ be a cyclic submodule of $P$ of order $p^4$. Then $V$ contains a core-free subgroup of order $p^3$ (indeed, the existence of such a subgroup is equivalent to the cyclicity of the dual module $V^*$, noting that $V$ is cyclic if and only if $V^*$ is cyclic, since $V$ is semisimple). In particular, $\a(G) \geqs 1$.   

Now assume that $\a(G)=1$. This means that all core-free subgroups $H$ with $b(G,H) \geqs 3$ must form a single conjugacy class of subgroups of $V$ of order $p^3$. We claim that $P=V$. Seeking a contradiction, suppose $V<P$ and let $A$ be a minimal normal $p$-subgroup of $G$ not contained in $V$. Since $c(G)=4$, it follows that $V=B \times C$ where $A \cong B$ as $G$-modules and so $A \times B$ is a cyclic module of rank $4$ with a core-free subgroup of order $p^3$ (not contained in $V$). Therefore, $\a(G) \geqs 2$ and we have reached a contradiction, so $P=V$ as claimed.   

Let us also observe that a core-free subgroup of $P$ of order $p^3$ is a maximal
core-free subgroup, whence $F(G) = P \times Q$ and every subgroup of $Q = O_{p'}(G)$ is normal in $G$. By Lemmas \ref{l:5}, \ref{l:6} and \ref{l:7},  one of the following holds:
\begin{itemize}\addtolength{\itemsep}{0.2\baselineskip}
\item[{\rm (I)}] $P = (C_p)^4$ is a minimal normal subgroup of $G$.
\item[{\rm (II)}] $P = A \times B$, where $A=(C_p)^2$ and $B = (C_p)^2$ are minimal normal subgroups of $G$ and non-isomorphic as $G$-modules.
\item[{\rm (III)}] $P = A \times B$, where $A=(C_p)^2$ and $B = (C_p)^2$ are minimal normal subgroups of $G$ that are isomorphic as $G$-modules and $G/C_G(A)$ is nonabelian.
\end{itemize}

Let $X$ be the set of core-free subgroups of $P$ of order $p^3$ and note that $|X|$ coincides with the number of subgroups of order $p$ which generate $P$ as a $G$-module. Since we are assuming $\a(G)=1$, it follows that $L$ acts regularly on $X$ and thus $|L|=|X|$. We now consider the three cases (I), (II) and (III).

First assume (I) holds, in which case $|X| = (p^4 -1)/(p-1)$. If $p \ne 2$, then $|X|$  is even and so $L$ contains an involution. However, any involution stabilizes some subgroup in $X$, which contradicts the fact that $L$ acts  regularly on $X$. Now assume $p=2$ and note that $|L|=|X|=15$, so $L = C_{15}$. Here the order $3$ subgroup of $L$ must act faithfully on $Q$ (otherwise we can construct a  core-free subgroup $H$ of order $12$ with $b(G,H) \geqs 3$, which gives $\a(G) \geqs 2$) and therefore $G$ is described as in part (iii)(a).

Conversely, suppose that $G$ satisfies the conditions in (i), (ii) and (iii)(a). We claim that $\a(G) = 1$. To see this, let $H$ be a nontrivial core-free subgroup of $G$ and set $D = H \cap F(G)$. Note that $L$ has a regular orbit on $F(G)$, so $b(G,H) = 2$ if $D=1$. By (ii), $G$ has a unique conjugacy class of core-free subgroups of order $8$, so it suffices to show that $b(G,H)=2$ if $|D|=2$ or $4$. If $|D|=2$, then $N_G(D) = F(G)$ and thus $H=D$ and $b(G,H)=2$. Now assume $|D|=4$. First observe that $F(G)$ has $35$ subgroups of order $4$, comprising three $L$-orbits (two of size $15$ and one of size $5$). It is straightforward to show that each orbit contains a pair of subgroups $D_1, D_2$ such that $D_1 \cap D_2 = 1$, whence $b(G,D) = 2$. So we may assume there exists an element $y \in N_L(D)$ of order $3$, in which case we need to consider core-free subgroups of the form $H = \langle D, z \rangle$, where $z \in yF(G)$ has order $3$.  Fix $w \in L$ such that $D \cap D^w=1$ and let $R = \la u \ra$ be a subgroup of $F(G)$ of odd prime order such that $y \not\in C_L(R)$. Then $H^w \cap H^u = 1$, so $b(G,H) =2$ and we conclude that $\a(G)=1$.  

Next suppose (II) holds and note that $|X| = (p+1)^2(p-1)$. First assume $p$ is odd and let $S$ be a Sylow $2$-subgroup of $L$. If $S$ contains an involution which has a common eigenvalue on $A$ and $B$, then there exists an involution normalizing a core-free subgroup of $P$ of order $p^3$ and we deduce that $\a(G) \geqs 2$. Therefore, we may assume $S$ contains a unique involution. By Theorem \ref{t:evenderived}, $S$ must be abelian and so $S$ is cyclic. This implies that $|S|$ divides $p^2-1$, which is incompatible with the fact that $L$ acts transitively on $X$.  

Now assume $p=2$, in which case $|X|=9$ and $L \leqs {\rm GL}_2(2)^2 = (S_3)^2$, so $L = C_3 \times C_3$. Suppose there exists a ``diagonal" $C_3$ subgroup of $L$ which centralizes $Q$. Then $P{:}L = A_4 \times A_4$ contains a diagonal subgroup $J \cong A_4$ with normal closure $PJ$. It follows that $b(G,J) \geqs 3$ and thus $\a(G) \geqs 2$. Therefore, $\a(G)=1$ only if the conditions in (iii)(b) are satisfied.

Conversely, suppose that $G$ satisfies the conditions in (i), (ii) and (iii)(b). Let $H$ be an arbitrary core-free subgroup of $G$ with $b(G,H) \geqs 3$ and note that $D = H \cap F(G) \ne 1$. In particular, if we can show that $|D| \ne 2,4$ then $\a(G) = 1$. If $|D|=2$, then $N_G(D)=F(G)$ and $b(G,H)=b(G,D)=2$, a contradiction. Now assume $|D|=4$. First observe that $F(G)$ has $33$ core-free subgroups of order $4$, which are permuted in
orbits of size $9,9,9,3$ and $3$ by $L$.  One computes that each orbit of size $9$ contains a pair of subgroups with trivial intersection, so if $D$ is contained in one of these orbits then $H=D$ and $b(G,H) = 2$. Now assume $D$ is contained in one of the orbits of size $3$, which means that $J=N_L(D)$ has order $3$. We can view $P$ as a direct sum of two irreducible and isomorphic $J$-modules, which means that the intersection of any two distinct $J$-invariant proper subgroups of $P$
is trivial. In particular, $D \cap D^w=1$ for some (in fact all) $w \in L \setminus{J}$. Now the conditions in (iii)(b) imply that $H$ acts nontrivially on $Q$ and we may fix $u \in Q$ such that $J \not\leqs C_L(u)$. Then $H^w \cap H^u=1$, so $b(G,H) = 2$ and we conclude that $\a(G) = 1$.

Finally, let us assume (III) holds. We show that $\a(G) \geqs 2$ in this case. Note that  
\[
|X| = \frac{p^4-1}{p-1} - (p+1)^2 =p(p^2-1)
\]
is the number of core-free subgroups of $P$ of order $p^3$. First assume $p=2$ and observe that $L=S_3$ since this is the only nonabelian group acting faithfully and irreducibly on $C_2 \times C_2$. Since every subgroup of $Q$ is normal, it follows that $L'$ centralizes $Q$. Now $L'$ normalizes five subgroups of $P$ of order $4$, only three of which are normal. Therefore, there is a core-free
subgroup $D$ of $P$ of order $4$ normalized by $L'$. Since the normal closure of $DL'$ is $PL'$, it follows that $b(G,DL') \geqs 3$ and thus $\a(G) \geqs 2$.

Now assume $p$ is odd. Since $p$ divides $|X|$, it follows that $L$ contains unipotent elements and thus $L = {\rm SL}_2(p)$ is the only possibility since $L$ acts irreducibly on a $2$-dimensional module. Therefore, $L'$ has even order and by applying Theorem \ref{t:evenderived} we conclude that $\a(G) \geqs 2$ (one could also argue directly in this case). 
\end{proof} 

\begin{rem} \label{r:examples}  
Referring to cases (a) and (b) in part (iii) of Theorem \ref{t:c4}, we find that the smallest groups with $\a(G) = 1$ are of the form
\[
((C_2)^4 \times C_7){:}C_{15} \;  \mbox{ and } \; ((C_2)^2 \times (C_2)^2 \times C_7){:}(C_{3} \times C_3)
\]
respectively.  
\end{rem} 

\begin{thm}\label{t:c3}
Suppose $E(G)=1$ and $c(G) = 3$. Then $\a(G) \geqs 1$, with equality if and only if 
all of the following conditions are satisfied:
\begin{itemize}\addtolength{\itemsep}{0.2\baselineskip}
\item[{\rm (i)}] $F(G) = P \times Q$, where $|P|=p^3$ and every subgroup of $Q = O_{p'}(G)$ is normal in $G$.
\item[{\rm (ii)}] $L$ acts regularly on the set of core-free subgroups of $P$ of order $p^2$, and in particular, $L \cap Z({\rm GL}_3(p))=1$. 
\item[{\rm (iii)}] One of the following holds:

\vspace{1mm}

\begin{itemize}\addtolength{\itemsep}{0.2\baselineskip}
\item[{\rm (a)}] $P = (C_p)^3$ is a minimal normal subgroup of $G$, $L = C_{p^2+p+1}$ and $p \not\equiv 1 \imod{3}$.
\item[{\rm (b)}] $P = A \times B$, where $A = (C_p)^2$ and $B = C_p$ are minimal normal subgroups of $G$. In addition, $L$ acts primitively on $A$ and it contains no nontrivial element that acts as a scalar on $A$ and acts trivially on $B \times Q$. Furthermore, if $p \geqs 3$ then
$p \equiv 3 \imod{4}$ and $|Q|$ is divisible by a prime $r \equiv 1 \imod{8}$. 
\item[{\rm (c)}] $P = A \times B \times C$ and $p \geqs 3$, where $A$, $B$ and $C$ are normal subgroups of $G$ that are pairwise non-isomorphic as $G$-modules.
In addition, $|Q|$ is divisible by an odd prime and $L$ contains no nontrivial element that acts as a scalar on $A \times B$ and acts trivially on $C \times Q$ (and similarly for any permutation of $A$, $B$ and $C$). 
\end{itemize}
\end{itemize}
\end{thm}

\begin{proof}   
Let $p$ be a prime with $c_p(G)=3$ and let $V$ be a subgroup of $F(G)$ of order $p^3$ that is cyclic as a $G$-module. Since $V$ is cyclic, it contains a core-free subgroup $H$ of order $p^2$ with normal closure $V$. Therefore, $b(G,H) \geqs 3$ and $\a(G) \geqs 1$. 
 
Now assume $\a(G) = 1$ and set $P=O_p(G)$. As in the proof of Theorem \ref{t:c4}, we deduce that $P = V$ and $F(G) = P \times Q$, where every subgroup of $Q = O_{p'}(G)$  is normal in $G$. In addition, we see that the possibilities for $P$ are as follows:
\begin{itemize}\addtolength{\itemsep}{0.2\baselineskip}
\item[{\rm (I)}] $P = (C_p)^3$ is a minimal normal subgroup of $G$.
\item[{\rm (II)}] $P = A \times B$, where $A = (C_p)^2$ and $B = C_p$ are minimal normal subgroups of $G$.
\item[{\rm (III)}] $P = A \times B \times C$, where $A, B$ and $C$ are normal subgroups of $G$ that are pairwise non-isomorphic as $G$-modules.   
\end{itemize}
Let $X$ be the set of core-free subgroups of $P$ of order $p^2$ and note that $|X|$ is the number of subgroups of $P$ of order $p$ which generate $P$ as a $G$-module. Since $\a(G)=1$, we deduce that $L$ acts regularly on $X$ and thus $G$ satisfies the conditions recorded in parts (i) and (ii) in the statement of the theorem. We now handle cases (I), (II) and (III) in turn.

First consider (I), in which case $|X| = p^2+p+1$. If $p \equiv 1 \imod{3}$ then $L$ contains an element of order $3$. But any element of order $3$ stabilizes a subgroup in $X$, which is incompatible with the fact that the action of $L$ on $X$ is regular, whence $p \not\equiv 1 \imod{3}$. Since the only irreducible subgroup of $\GL_3(p)$ of order $p^2 + p + 1$ is a cyclic maximal torus, we deduce that (iii)(a) holds. Note that since $L$ acts fixed point freely on the set of nontrivial elements of $P$, any core-free subgroup of $G$ is either contained in $P$ or intersects $F(G)$ trivially. It is clear that the only such subgroups $H$ with $b(G,H) \geqs 3$ are those contained in $P$ with order $p^2$, whence $\a(G) = 1$ since any two of these subgroups are conjugate by (ii).  

Now let us turn to (II) and note that $|L| = |X| = p^2-1$. If $p=2$, then $L = C_3$ and it is straightforward to see that the only core-free subgroups $H$ with $b(G,H) \geqs 3$ are the three core-free subgroups of $P$ of order $4$, all of which are conjugate by (ii). Therefore, $\a(G)=1$ in this case. 

Now assume $p$ is odd. Let $S$ be a Sylow $2$-subgroup of $L$ and note that $S$ is abelian by Theorem \ref{t:evenderived}. Thus $L = M{:}S$,  where $M$ is abelian of odd order. It is now straightforward to see that $\alpha(G)=1$ if and only if all of the conditions in parts (i), (ii) and (iii)(b) in the statement of the theorem are satisfied. Indeed, we have already noted that the conditions in (i) and (ii) are equivalent to the existence of a unique conjugacy class in $G$ of core-free subgroups of $P$ of order $p^2$. So we need only show that there are no other core-free subgroups $H$
of $G$ with $b(G,H) \geqs 3$. 

Seeking a contradiction, suppose $G$ has such a subgroup $H$, which is not an order $p^2$ subgroup of $P$. Set $D = H \cap F(G)$. Since $L$ acts regularly on $X$, $L$ has a regular orbit on $F(G)$ and thus $D \ne 1$, which means that $|D|=p$. First assume that the normal closure of $D$ is $P$. Then since $L$ acts regularly on $X$ (and hence, by duality, on the set of subgroups of $P$ of order $p$, which generate $P$ as a $G$-module), it follows that $N_L(D)=1$, $H=D$ and $b(G,H)=2$. The only other possibility is that the normal
closure of $D$ is $A$. Since $L$ acts regularly on $X$ and $|L|=p^2-1$, the only elements of $L$ which normalize such a subgroup $D$ must induce scalars on $A$ and so the corresponding subgroup $J$ of $L$ acts faithfully on $B \times Q$. Moreover, $J$ is abelian and so $J$ has a regular orbit on $W=B \times Q$. One can reduce to the case where $C_W(J)=1$.  Then since $J$ is abelian and has order prime to $p$, it follows that $H^1(J, F(G))=0$ and so we may assume that $J$ is a (normal)
subgroup of $L$ and 
$H=D{:}J$. Then there exists an element $x \in L$ with $C^x \ne C$ and an element $v \in W$ generating a regular $J$-orbit. It is now straightforward to show that $H^x \cap H^v=1$, so we have reached a contradiction and thus $\a(G)=1$.  

Note that the transitivity of $L$ on $X$ implies that $L$ acts transitively on the set of lines in $A$.  In particular, $L$ must act primitively on $A$. Moreover, we  claim that $S$ is cyclic. If not, then $S$ contains an elementary abelian subgroup of order $4$. But we have already noted that $S$ does not contain an involution acting as a scalar on $P = A \times B$. Any involution in $S$ which does not act as a scalar on $A$ must fix a subgroup in $X$, which contradicts the fact that $L$ acts regularly on $X$. Therefore, $S$ contains a unique involution and thus $S$ is cyclic (recall that $S$ is abelian). Since $|S|=(p^2-1)_2 > (p-1)_2$, it follows that $S$ acts faithfully on $A$ and the involution in $S$ is trivial on $B$.   However, $S$ acts faithfully on $B \times Q$ and so on $Q$.  Since
$|S| \geqs 8$,  it follows that $|Q|$ must be divisible by a prime $r$ with $r \equiv 1 \imod{8}$, as recorded in (iii)(b). 
 
Finally we consider (III). Here $p$ is odd and $|L| = |X| = (p-1)^2$.  Since $L$ acts regularly on $X$, it does not contain any nontrivial scalars. If $L$ contains an element $x \ne 1$ acting as a scalar on $A \times B$ and centralizing $C \times Q$,
then $J = \langle D, x \rangle$ has normal closure $\langle A \times B, x \rangle$,  where $D$ is a diagonal subgroup of $A \times B$. Therefore, $b(G, J) \geqs 3$ and $\a(G) \geqs 2$. This shows that $\a(G) = 1$ only if the conditions in (iii)(c) are satisfied. 

Note that if (III) holds and $\a(G)=1$, then the Sylow $2$-subgroup of $L$ is not cyclic (indeed, if this subgroup is cyclic then $|L|_2 \leqs (p-1)_2$, but we have
$|L|_2 = ((p - 1)_2)^2$). Therefore, $L$ contains an involution acting as a scalar on $A \times B$ (or some other product) and so must act nontrivially on $C \times Q$.   In particular, $L$ contains an involution acting as inversion on $A \times B$ and centralizing $C$ (up to permutations of $A$, $B$ and $C$). Therefore, this involution must act nontrivially on $Q$ and so $|Q|$ is divisible by an odd prime. 

Conversely, suppose that $G$ satisfies the conditions in (i), (ii) and (iii)(c). Let $H$ be a core-free subgroup of $G$ with $b(G,H) \geqs 3$ and set $D = H \cap P$, noting that $D \ne 1$ since $L$ has a regular orbit on $P$. If $|D|=p^2$ then $N_L(D)=1$ and $G$ has a unique class of such subgroups. Now assume $|D|=p$, in which case the normal closure of $D$ has order $p^2$ or $p^3$. Since $L$ acts regularly on $X$, it also acts regularly on the set of subgroups of order $p$ which generate $P$ as $G$-module. So if the normal closure of $D$ has order $p^3$, then $H=D$ and $b(G,H)=2$, a contradiction.  Now assume the normal closure of $D$ has order $p^2$, say $A \times B$. Then $H$ induces scalars on $A \times B$ and thus $H/D$ acts faithfully on $C \times Q$. Write $H = DH_0$ for some complement $H_0$ to $D$.   

Let $W = C \times Q  = [H,W] \times C_W(H)$.   By assumption,  $H_0$ acts faithfully on $W$ and so on $[H,W]$.  As in the above analysis of case (II), we may assume that $C_W(H)=1$ and $H_0$ is contained in $L$. Since $H_0$ is abelian, it  has a regular orbit on $[H,W]$, say $v \in [H,W]$ is contained in such an orbit and fix  $x \in L$ with $D \cap D^x=1$. Then $H^v \cap H^x=1$ and we conclude that 
$\a(G)=1$. 
\end{proof}

\begin{rem}
Let us consider the groups $G$ in Theorem \ref{t:c3} with $\a(G)=1$, which arise in cases (a), (b) and (c) in part (iii). A minimal example in (iii)(a) is to take $Q=1$, which yields the primitive affine group $G = (C_p)^3{:}C_{p^2+p+1} < {\rm AGL}_3(p)$.  
In (iii)(b), if $p = 2$, we can take $G=A_4 \times C_2$.  More generally, we can take $L = C_{p^2-1}$ acting irreducibly and faithfully on $A$, with the Sylow $2$-subgroup of $L$ acting faithfully on $Q = C_r$.  Similarly, in (iii)(c) we can take $F(G) = A \times B \times C \times Q$  so that $L=C_{p-1} \times C_{p-1}$ acts faithfully on $A \times B$ and the subgroup of $L$ acting as scalars on $A \times B$ acts faithfully on $Q = C_r$, where $r \ne p$ is a prime and $p-1$ divides $r-1$. For example, if $p=3$ and $r=5$, then $G=((C_3)^3 \times C_5){:}(C_2)^2$ and $\a(G)=1$ (here $G$ is isomorphic to $D_{30} \times S_3 \times C_3$).
\end{rem}

\begin{thm}\label{t:c1}
Suppose $E(G)=1$ and $c(G) = 1$. Then $\a(G) = 0$.
\end{thm}

\begin{proof}  Since $E(G)=1$ and $c(G)=1$,  $L$ is abelian and every subgroup of $F(G)$ is normal in $G$. Now apply Corollary \ref{c:9}.  
\end{proof}

By combining Theorems \ref{t:c5}, \ref{t:c4}, \ref{t:c3} and \ref{t:c1}, we obtain the following corollary.

\begin{cor}\label{c:eg1}
Suppose $E(G)=1$, $c(G) \ne 2$ and $\a(G) \leqs 1$. Then $G$ is solvable and either $c(G) \in \{3,4\}$ and $\a(G)=1$, or $c(G)=1$ and $\a(G) = 0$. 
\end{cor}

\section{Proof of Corollaries \ref{c:ns} and \ref{c:prim}}\label{s:cor}

In this section we prove Corollaries \ref{c:ns} and \ref{c:prim}. As usual, we assume $G$ is a finite group with $\Phi(G)=1$ and we write $G = F(G){:}L$ as in \eqref{e:FGL}.

\begin{proof}[Proof of Corollary \ref{c:ns}]
Let $G$ be a finite nonsolvable group. If $E(G) \ne 1$ then Theorem \ref{t:main1} applies and thus $G$ has the form given in part (i) of the corollary (note that the possibilities for $T$ are recorded in Theorem \ref{t:asb2}). 

Now assume $E(G)=1$ and define $c(G)$ and $r(G)$ as before. Since $G$ is nonsolvable it follows that $r(G) \geqs 2$, while Corollary \ref{c:eg1} yields $c(G)=2$. Therefore, $r(G) = c(G)=2$. If $L'$ has odd order, then $L = J{:}S$, where $J = O(L)$ is the largest odd order normal subgroup of $L$ and $S$ is an abelian $2$-group. In particular, $G$ is solvable and so we may assume $L'$ contains an involution. We now conclude via Theorem \ref{t:evenderived}, noting that part (iv)(d) holds since $G$ is nonsolvable.
\end{proof}

Now let us turn to Corollary \ref{c:prim}. We first establish a more general result on regular orbits.

\begin{prop}\label{p:regorbit}  
Suppose $c(G) \leqs 2$, $L'$ has odd order and $L$ has no regular orbit on $F(G)$. 
Then $F(G) = P \times Q$, where $P = O_2(G)$ has order $4$ and $L = S_3 \times L_0$ so that $L_0 = C_L(P)$ and $S_3$ acts faithfully on $P$ and trivially on $Q$.  
In particular, $G = S_4 \times G_0$ with $G_0 = Q{:}L_0$. 
\end{prop}

\begin{proof}  
Clearly if $G = S_4 \times G_0$, then the centralizer of any element
of $F(G)$ contains an involution and so $L$ does not have a regular orbit on $F(G)$. 
For the remainder, let us assume $L$ has no regular orbit on $F(G)$.

Let $p$ be an odd prime and set $R=L/C_L(O_p(G))$. We claim that $R$ has a regular orbit on $O_p(G)$. This is clear if $R$ is abelian, so let us assume $R$ is nonabelian.   This implies that $O_p(G)$ has order $p^2$ and so we may view $R$ as a subgroup of ${\rm GL}_2(p)$. We also observe that $O(R)$ is  cyclic and the Sylow $2$-subgroup of $R$ is abelian. Since we are assuming $L'$ has odd order, it follows that $R'$ also has odd order and thus $|R|$ is indivisible by $p$. This implies that either 
\begin{itemize}\addtolength{\itemsep}{0.2\baselineskip}
\item[{\rm (a)}] $R$ acts imprimitively on $O_p(G)$ and is therefore contained in the normalizer of a split maximal torus in ${\rm GL}_2(p)$; or 
\item[{\rm (b)}] $R$ acts primitively and is contained in the normalizer of a nonsplit maximal torus. 
\end{itemize}

First assume (a) holds and write $R \leqs N$, where $N = N_{{\rm GL}_2(p)}(T) = C_{p-1} \wr S_2$ and $T = (C_{p-1})^2$ is a split maximal torus of ${\rm GL}_2(p)$. Here $O_p(G)$ has two lines whose union is $R$-invariant and we claim that $R$ has a regular orbit on the remaining $p^2 - (2p-1) = (p-1)^2$ elements. These points lie on $p-1$ distinct lines. Suppose $y \in N$ acts trivially on
one of these lines. Then $y \in N \setminus T$ is an involution and $y$ acts trivially on a unique line. It follows that $R$ has no regular orbit on $O_p(G)$ if and only if $R$ contains all such involutions. However, the group generated by these involutions has a nonabelian Sylow $2$-subgroup, a contradiction.

A very similar argument applies in case (b). Here $R \leqs N$ with $N = N_{{\rm GL}_2(p)}(T) = C_{p^2-1}.C_2$. There are $p+1$ lines in total and the only elements in $N$ that act trivially on a line are involutions in $N \setminus T$. So if $R$ has no regular orbit, then all such involutions must be in $R$ and once again we reach a contradiction since the Sylow $2$-subgroup of the group generated by these involutions is nonabelian. This justifies the claim. 

Now write $F(G) = O_2(G) \times Q$ and let $K = C_L(Q)$ be the subgroup of $L$ acting trivially on $Q$. By the previous claim, $L/C_L(O_p(G))$ has a regular orbit on $O_p(G)$ for each odd prime $p$, so there exists $q \in Q$ such that $C_L(q) = K$. Now $K$ acts faithfully on $O_2(G)$, and it has a regular orbit unless $K = S_3$ and $O_2(G) = C_2 \times C_2$. So we may assume we are in the latter situation, in which case $K$ does not have a regular orbit. Since $K$ is normal in $L$, and since $K$ is its own automorphism group, we deduce that $L = K \times C_L(P)$.   Therefore
\[
G = (O_2(G){:}K) \times G_0 = S_4 \times G_0
\]
and the result follows. 
\end{proof}

\begin{cor}\label{c:regorbit}   
Suppose $c(G) \leqs 2$, $L'$ has odd order and $L$ has no regular orbit on $F(G)$. 
Then $\alpha(G) \geqs 1$, with equality if and only if $G = S_4 \times A$ and $A$ is an abelian group of odd order.
\end{cor}

\begin{proof}  
Since $L$ does not have a regular orbit on $F(G)$, it follows that $b(G,L) \geqs 3$ and thus $\a(G) \geqs 1$. In addition, if $G$ has the prescribed form, then it is easy to see that $L$ is the only core-free subgroup $H$ of $G$ (up to conjugacy) with $b(G,H) \geqs 3$ and thus $\a(G) = 1$.   

To complete the proof, suppose $\alpha(G) =1$. By Proposition \ref{p:regorbit}, we have $O_2(G) = C_2 \times C_2$ and $L$ has a subgroup $H = S_3$ with $b(G,H) \geqs 3$. Since $b(G,L) \geqs 3$, it follows that $H$ and $L$ are conjugate and thus  $G=F(G){:}H = S_4 \times A$, where $A$ is a complement to $O_2(G)$ in $F(G)$. Therefore, $A$ is abelian and we note that $c_2(G) \geqs 3$ if $|A|$ is even, so the hypothesis $c(G) \leqs 2$ implies that $|A|$ is odd.  
\end{proof}  

We will also need the following proposition, which handles a family of groups with a primitive permutation representation of degree $p^2$ for some prime $p$.

\begin{prop}\label{p:prim2}
Let $G = F(G){:}L$ be a primitive affine group, where $F(G) = (C_p)^2$ and $L \leqs {\rm GL}_2(p)$ is irreducible. Assume $L'$ has odd order. 
\begin{itemize}\addtolength{\itemsep}{0.2\baselineskip}
\item[{\rm (i)}] If $p=2$, then $\alpha(G)=0$ if $L=C_3$, and $\alpha(G)=1$
if $L=S_3$.
\item[{\rm (ii)}] Suppose $L$ is abelian. Then $\alpha(G)=0$ if and only if $L \cap Z({\rm GL}_2(p)) = 1$, and $\alpha(G)=1$ if and only if $p \equiv 3 \imod{4}$ and $L = C_{2(p+1)}$.
\item[{\rm (iii)}] Suppose $L$ is nonabelian and acts imprimitively on $F(G)$. Then $\alpha(G) \geqs 1$, with equality if and only if $L = D_{2r}$ and $r$ is an odd prime divisor of $p-1$.  
\item[{\rm (iv)}] Suppose $L$ is nonabelian and acts primitively on $F(G)$. Then $\alpha(G) = 0$ if and only if $L \cap Z({\rm GL}_2(p))=1$, and $\alpha(G) = 1$ if and only if $p \equiv 3 \imod{4}$, $|L|=2(p+1)$ and $Z(L) = C_2$.  
\end{itemize}
\end{prop}

\begin{proof}  
The result is clear if $p=2$, so let us assume $p$ is odd and $H$ is a nontrivial core-free subgroup of $G$. If $H \cap F(G) = 1$ then $b(G,H) = 2$ since $L$ has a regular orbit on $F(G)$ by Proposition  \ref{p:regorbit}. In particular, if $b(G,H) \geqs 3$ then $H \cap F(G)$ has order $p$. 

Since $L'$ has odd order, it follows that $|L|$ is indivisible by $p$.  In particular, $H^1(L,F(G))=0$ and thus any two complements of $F(G)$ in $G$ are conjugate. Let $Z = L \cap Z({\rm GL}_2(p)) \leqs C_{p-1}$ be the group of scalars in $L$.

First assume $L$ is abelian and let $D$ be a subgroup of order $p$. Then $b(G,D) = 2$ and $N_L(D) = Z$. If $Z \ne 1$ then $b(G,H) \geqs 3$ for any subgroup of the form $H = D{:}Y$ with $1 \ne Y \leqs Z$, whence $\a(G) \geqs 2$ if $Z$ has composite order. On the other hand, if $Z = 1$ then $\a(G) = 0$. Now assume $Z$ has prime order, so $\a(G) \geqs 1$. Clearly, $\alpha(G)=1$ only if $L/Z$ acts regularly on the set of order $p$ subgroups of $F(G)$, in which case $|L|=(p+1)|Z|$ is even.  Since any involution in $L$ fixes a line, this implies that $|Z|=2$ and $L$ is cyclic of order $2(p+1)$. In addition, since $Z$ is the stabilizer of any line, this forces $p \equiv 3 \imod{4}$ and we conclude that (ii) holds. 

For the remainder we may assume $L$ is nonabelian. We divide the analysis according to whether or not $L$ acts imprimitively (as a linear group) on $F(G)$.

First assume $L$ acts imprimitively on $F(G)$. Then $L$
has a normal abelian subgroup $L_0 \leqs (C_{p-1})^2$ of index $2$. If $K \leqs L_0$ is any nontrivial normal subgroup of $L$, then clearly $b(G,DK) \geqs 3$, where $D = C_p$ is one of the two lines in $F(G)$ fixed by $L_0$. This implies that $\a(G) \geqs 1$, with equality only if $L_0$ contains no proper nontrivial normal subgroups of $L$. This observation, together with the fact that $L'$ has odd order, implies that $L_0 = C_r$ for some odd prime $r$ and we deduce that (iii) holds. Here it is straightforward to see that the only core-free subgroups $H$ of $G$ with $b(G,H) \geqs 3$ are of the form $D{:}L_0$, where $D = C_p$ is one of the two lines fixed by $L_0$.   

Now let us assume $L$ is nonabelian and primitive on $F(G)$. Since $L'$ has odd order, this implies that $L$ is contained in the normalizer of a nonsplit torus in ${\rm GL}_2(p)$ and $L$ has an abelian subgroup $L_0 \leqs C_{p^2-1}$ of index $2$. 

If $Z=1$, then $L_0$ has odd order and no element of $L_0$ fixes a line, which implies that $L$ is a dihedral group and $|L|$ divides $p+1$. Therefore, the stabilizer of a line in $L$ has order at most $2$. Since the involutions in $L$ are self-centralizing, we deduce that the stabilizer of a pair of conjugate lines is trivial and thus $b(G,H) =2$ for every nontrivial core-free subgroup $H$ of $G$. In particular, we conclude that $\a(G)=0$.   

Now assume $Z \ne 1$, so $\a(G) \geqs 1$. Let us assume $\a(G) = 1$. Then $Z$ has prime order and $L/Z$ acts regularly on the set of lines of $F(G)$, so $|L/Z| = p+1$. We also deduce that $D{:}Z$ is a maximal core-free subgroup of $G$ for any subgroup $D$ of order $p$, so $Z$ must be the stabilizer of every line in $F(G)$. Since $L$ has even order, this implies that $|Z|=2$ and $|L|=2(p+1)$. Moreover, every involution in $L$ must be in $Z$ and no element in $L$ of order $4$ fixes a  line. Therefore, the Sylow $2$-subgroup of $L$ is cyclic of order $4$
and thus $p \equiv 3 \imod{4}$ as in (iv).  
\end{proof} 

\begin{proof}[Proof of Corollary \ref{c:prim}]
Let $G$ be a finite group with socle $N$ and a core-free maximal subgroup. Then $N = S^d$ for some simple group $S$ and positive integer $d$. If $S$ is nonabelian, then Theorem \ref{t:main1} implies that $\a(G) \geqs 1$, with equality if and only if $G = N = S$ and $\a(S)=1$, so we may assume $S = C_p$ is abelian. Then $N = (C_p)^d = F(G)$ and $G = F(G){:}L \leqs {\rm AGL}_d(p)$, where $F(G)$ is the unique minimal normal subgroup of $G$. 

If $d=1$ then $c(G)=1$ and thus $\a(G) = 0$ by Theorem \ref{t:c1}. On the other hand, if $d \geqs 5$ then $c(G) \geqs 5$ and thus $\a(G) \geqs 2$ by Theorem \ref{t:c5}. Now assume $d \in \{2,3,4\}$. If $d=4$ then Theorem \ref{t:c4} implies that $\a(G) \geqs 2$, while for $d=3$ we can appeal to Theorem \ref{t:c3}, which shows that $\a(G) = 1$ if and only if $L = C_{p^2+p+1}$ and $p \not\equiv 1 \imod{3}$. 

Finally, let us assume $d=2$, in which case $r(G) = c(G) = 2$. If $L'$ has even order, then Theorem \ref{t:evenderived} implies that $\a(G) \geqs 1$, with equality if and only if $p \equiv 3 \imod{4}$ and either $L = 2.D_{p+1}$ has a unique involution, or $(p,L) = (11,{\rm SL}_2(3))$, $(23,2.S_4)$ or $(59, {\rm SL}_2(5))$. These cases are all included in part (iii)(d) of Corollary \ref{c:prim}. Finally, if $L'$ has odd order then we apply Proposition \ref{p:prim2}.
\end{proof}

\section{$E(G) = 1$, $c(G) = 2$} \label{s:cle2} 

In this final section we consider the strongly base-two groups with $\Phi(G) = E(G) = 1$ and $c(G)=2$. Recall that $r(G)$ denotes the maximal rank of a minimal normal subgroup of $G$. Since $r(G) \leqs c(G)$, we see that every minimal normal subgroup of $G$ has rank $1$ or $2$. As usual, write $G = F(G){:}L$ as in \eqref{e:FGL}. 

\begin{rem}
Recall that if $c(G) \geqs 3$, then one can always find a core-free subgroup $H$ of $F(G)$ with $b(G,H) \geqs 3$. But this need not be the case when $c(G) = 2$. Indeed, we may have $b(G,H) = 2$ for every nontrivial core-free subgroup $H$, which means that $\a(G)=0$. 
\end{rem}

We have already handled several special cases in this setting. For instance, the nonsolvable strongly base-two groups were determined in Corollary \ref{c:ns}, while the examples with a core-free maximal subgroup are recorded in Corollary \ref{c:prim}. The latter result shows that we can find groups with $c(G) = 2$ and $\a(G) = 0$ or $1$, and further examples can be constructed by combining these groups. Let us also observe that if $L'$ has even order, then $r(G)=2$ and Theorem \ref{t:evenderived} classifies the groups with $\a(G) \leqs 1$. One can also construct examples where $|L'|$ is odd and $\a(G)=1$, analogous to those in Theorem \ref{t:evenderived}.

For the remainder, let us assume $|L'|$ is odd and let $O(L)$ be the largest odd order normal subgroup of $L$. Then $O(L)$ is abelian and $L = O(L){:}S$ for some abelian $2$-group $S$. Note that the condition $c(G)=2$ implies that if $G$ has a minimal normal subgroup $P = (C_p)^2$ on which $G$ acts as a nonabelian group, then $P = O_p(G)$. 

In view of Corollary \ref{c:regorbit}, we may assume that $L$ has a regular orbit on $F(G)$. As a consequence, it is straightforward to show that there are no core-free subgroups $H$ of $G$ such that $b(G,H) \geqs 3$ and $H \cap F(G)=1$.  

The following result gives necessary and sufficient conditions for $\a(G)=0$. In view of Theorem \ref{t:evenderived}, note that $\a(G)=0$ only if the derived subgroup of $L$ has odd order. 
 
\begin{thm}\label{t:a=0} 
Suppose $E(G)=1$, $c(G)=2$ and $L'$ has odd order. Then $\alpha(G) = 0$ if and  only if all of the following conditions are satisfied:
\begin{itemize}\addtolength{\itemsep}{0.2\baselineskip} 
\item[{\rm (i)}] $L$ has a regular orbit on $F(G)$.
\item[{\rm(ii)}]  If $H$ is a core-free subgroup of $F(G)$, then $H \cap H^x =1 $ for some $x \in L$.
\item[{\rm (iii)}] $L$ does not contain a subgroup $J = \la y \ra$ of prime order such that $F(G) = V \times C_{F(G)}(y)$, where 
\[
V = [y, F(G)] = V_1 \times \cdots \times V_k
\]
and the following conditions are satisfied:

\vspace{1mm}

\begin{itemize}\addtolength{\itemsep}{0.2\baselineskip} 
\item[{\rm (a)}] Each $V_i = (C_{p_i})^2$ is a cyclic $G$-module.
\item[{\rm (b)}] $y$ acts as a nontrivial scalar on each $V_i$, unless
$V_i$ is imprimitive as a $G$-module and $y$ acts nontrivially on $V_i$, stabilizing  each of the two lines permuted by $G$. 
\end{itemize}
\item[{\rm(iv)}] $L$ does not contain a subgroup $J = S_3$ such that 
\[
V=[J,F(G)] = O_2(G) \times U,
\]
where $U = U_1 \times \cdots \times U_k$ with each $U_i = (C_{p_i})^2$ a cyclic $G$-module, and $J$ acts faithfully on $O_2(G) = (C_2)^2$ and by inversion on $U$.
\end{itemize}
\end{thm}

\begin{proof}  The necessity of (i) is clear (otherwise,  $b(G,L) \geqs 3$). The condition in (ii) is equivalent to the property that $b(G,H) \leqs 2$ for every core-free subgroup $H$ of $G$ contained in $F(G)$. 

Let us assume there exists an element $y \in L$ satisfying (iii).  
Let $U_i = C_{p_i}$ be a $y$-invariant subgroup of $V_i$. Note that if $y$ acts as a scalar on $V_i$, then every subgroup of $V_i$ is $y$-invariant, otherwise there are exactly two choices for $U_i$ (corresponding to the two lines in $V_i$ permuted by $G$). Set $U = U_1 \times \cdots \times U_k$ and let $Y$ be the normal closure of $\la y \ra$ as a subgroup of $L$. Note that each nontrivial element of $Y$ either acts as a scalar on each $V_i$, or it preserves the two lines in $V_i$ permuted by $G$ and acts trivially on $C_{F(G)}(y)$. Therefore, 
$Y$ normalizes $U$ and we deduce that $VY$ is the normal closure of $H = UY$ in $G$. Since $|H|^2 > |VY|$ we conclude that $b(G,H) \geqs 3$ and thus $\a(G) \geqs 1$.

Now suppose that there exists a subgroup $J$ as in (iv). Let $W_i \leqs U_i$ be a subgroup of order $p_i$ that is not normal in $G$ and set $W = W_1 \times \cdots \times W_k$ and $H = WJ$, noting that $VJ$ is the normal closure of $H$. Then
$|H|^2 = 36|U| > 24|U| = |VJ|$ and thus $b(G,H) \geqs 3$.  

We have now shown that if $\a(G)=0$ then all of the conditions in (i)-(iv) are satisfied. So it remains to show that the converse holds. To do this, let us assume (i)-(iv) holds and suppose $\a(G) \geqs 1$. Let $H$ be a maximal core-free subgroup of
$G$ with $b(G,H) \geqs 3$ and set $X = H \cap F(G)$, noting that $X$ is a nontrivial proper subgroup of $H$ by conditions (i) and (ii). Let $V$ be the normal closure of $X$ and write 
$V = V_1 \times \cdots \times V_k$ and $X = X_1 \times \cdots \times X_k$, where $V_i = (C_{p_i})^2$ for some prime $p_i$ and $X_i = X \cap V_i = C_{p_i}$ is normalized by $H$.

Suppose $p_i$ is odd. Then since the group induced by the action of $L$ on $V_i$ has a derived subgroup of odd order, it follows that every element in $L$ of order $p_i$ acts trivially on $V_i$. In particular, $H$ acts completely reducibly on $V_i$.     And for $p_i=2$ we note that $|V_i|=4$ and $|H/C_H(V_i)| \leqs 2$.  

Next observe that if $V_i$ (as an $L$-module) is either reducible, or irreducible and primitive, then the intersection of the stabilizers of distinct $L$-conjugate lines in $V_i$ consists of scalars (indeed, if the action of $L$ on $V_i$ is abelian, then the stabilizer of any line that is not $L$-invariant just consists of scalars). On the other hand, if $V_i$ is irreducible and imprimitive, then either the pair of conjugate lines form an $L$-orbit, or the intersection of the stabilizers consists of scalars. By (ii), we can choose $x \in L$ so
that $X_i^x \ne X_i$ for all $i$, which implies that $H \cap H^x \cap F(G)=1$. In addition, the above remarks show that if $h \in H \cap H^x$, then either $h$ acts as a scalar on $V_i$, or $V_i$ is imprimitive as an $L$-module and $h$ stabilizes each of the two lines permuted by $L$. If we now write $F(G) = V \times W$, where $W$ is a normal subgroup of $G$, then there are two cases to consider.  

First assume that $H/C_H(W)$  has a regular orbit on $W$, which means that $C_H(w) = C_H(W)$ for some $w \in W$. Then $H^{xw} \cap H \cap F(G) = 1$ and each element in $H^{xw} \cap H$ centralizes $W$ and acts as a scalar on each $V_i$. By (iii), it follows that $H^{xw} \cap H = 1$, a contradiction.

Now suppose that $H/C_H(W)$ does not have a regular orbit on $W$.  By Proposition \ref{p:regorbit}, this implies that $W = O_2(G) \times U$ and $H/C_H(V \times U)$ acts on $O_2(G)$ as $S_3$.  By a further application of Proposition \ref{p:regorbit}, it follows that there exists $u \in U$ such that
$C_H(u)$ centralizes $U$ (that is, $H/C_H(U)$ has a regular orbit on $U$). Then by arguing as in the previous paragraph, we see that $H^{xu} \cap H$ acts as scalars on each $V_i$ and contains an $S_3$ subgroup whose elements of order $3$ centralize $V$. This contradicts (iv) and completes the proof. 
\end{proof} 

\begin{rem}\label{r:J}   
Let us note that we may modify conditions (iii) and (iv) in Theorem \ref{t:a=0} so that the subgroup $J$ is normal in $L$. This is clear in case (iv) and indeed we note that $L =  JC_L(J)$. Now consider (iii). If $y$ acts as a scalar on each 
$V_i$, then $y$ is central in $L$ and thus $J = \la y \ra$ is normal. So let us assume  $y$ is not central in $L$. Then $\la y' \ra$ satisfies the conditions in (iii) for any element $y' \in [y,L] \leqs  L'  \leqs O(L)$ of prime order and either $y'$ acts trivially on $V_i$, or $V_i$ is imprimitive as a $G$-module. Let $Y$ be the normal closure of $\la y' \ra$ in $L$. If $x \in L$, then $x^2$ centralizes $Y$ and so $L$ acts on $Y$ as an elementary abelian $2$-group. In particular, $Y$ is a completely reducible $L$-module with $1$-dimensional summands and therefore we can choose $y'$ so that $\langle y' \rangle$ is normal in $L$.  
\end{rem} 

By inspecting the proof of Theorem \ref{t:a=0}, together with the relevant proofs in Sections \ref{s:t3} and \ref{s:c4}, we obtain the following characterization of the groups with $\a(G) \geqs 1$.

\begin{thm}\label{t:ql}  
Let $G$ be a finite group with $\Phi(G) = 1$.  Then $\a(G) \geqs 1$ if and only if
there exists a core-free subgroup $H$ of $G$ with $|H|^2 \geqs |N|$ and $b(G,H) \geqs 3$,  where $N$ is the normal closure of $H$ in $G$. 
\end{thm}

The next result shows that $\a(G)=1$ only if $b(G,H) =2$ for every nontrivial core-free subgroup $H$ of $G$ contained in $F(G)$.

\begin{prop}\label{p:subgroupsofF}  
Suppose $E(G)=1$, $c(G)=2$ and $L'$ has odd order. If $F(G)$ contains a core-free subgroup $H$ of $G$ 
with $b(G,H) \geqs 3$, then $\alpha(G) \geqs 2$.
\end{prop}

\begin{proof}   
Seeking a contradiction, let us assume $\alpha(G)=1$ so $H$ is a maximal core-free subgroup of $G$. In particular, $N_L(H)=1$ since $HN_L(H)$ is also core-free. Also note that if $K$ is a proper subgroup of $H$, then $K \cap K^x = 1$ for some $x \in L$. 

Let $V \leqs F(G)$ be the normal closure of $H$ in $G$ and write $V=V_0 \times V_1 \times \cdots \times V_k$, where $|V_0| \in \{1,4\}$ and for $i \geqs 1$ we have $V_i = (C_{p_i})^2$ for some odd prime $p_i$. Set $W_i = V_i \cap H$ and let $K_i$ be the group induced by the action of $N_L(W_i)$ on $W_i$. In addition, let $J \leqs S_3$ be the group induced by the action of $L$ on $V_0$. There are two cases to consider.

First assume that $J \ne S_3$. We claim that for each $i \geqs 0$, there exists a proper nontrivial subgroup $U_i$ of $V_i$ with $U_i \ne W_i$ and $N_L(U_i)=N_L(W_i)$. If $|K_i| \geqs 3$ then $L$ acts imprimitively on $V_i$ and the claim clearly holds in this case. For $|K_i|=2$ we can take $U_i$ to be the other eigenspace of $K_i$ on $V_i$, and if $K_i$ is trivial then we can choose 
any line $U_i \ne W_i$ in the $L$-orbit of $W_i$ in $V_i$. This justifies the claim and we set $U = U_0 \times \cdots \times U_k$. Since $N_L(U_i) = N_L(W_i)$ for all $i$, it follows that $b(G,U)=b(G,H)$. Therefore, the condition $\a(G) = 1$ implies that  $H$ and $U$ are $L$-conjugate and so there is some element $x \in L$ with $U_i=W_i^x$ for all $i$. But this gives $H \cap H^x =1$, a contradiction.

Finally, suppose $J = S_3$ and fix $y \in L'$ acting as an element of order $3$ on $V_0$. Set $U_0=W_0$. As in the previous case, for $i \geqs 1$ we choose $U_i$ to be a line in $V_i$ such that $N_L(U_i) = N_L(W_i)$. In addition, we claim that we can choose the lines so that $U_i \ne W_i^y$ for all $i \geqs 1$. This is clear if $y$ acts trivially on $V_i$, and more generally if $y$ stabilizes $W_i$. If $|K_i|=2$,  then $y$ centralizes $K_i$ and therefore preserves each of the two eigenspaces of $K_i$ on $V_i$. And if  $K_i=1$, then we take $U_i = W_i^{y^{-1}} \ne W_i^y$.  

Let  $U = U_0 \times \cdots \times U_k$. Since $N_L(U_i) = N_L(W_i)$ for all $i$, it follows that $b(G,U) = b(G,H)$ and so $U = H^x$ for some $x \in L$. Therefore, $W_i^x = U_i$ for all $i \geqs 1$ and thus $x$ must normalize $W_0$.
Setting $z = xy^{-1}$, we have $W_0^z = W_0^{y^{-1}} \ne W_0$ and also $W_i^z  = U_i^{y^{-1}} \ne W_i$ for $i \geqs 1$. This implies that $H \cap H^z=1$ and once again we have reached a contradiction. 
\end{proof}

To conclude, let us briefly explain how one can use Theorem \ref{t:a=0} and Proposition \ref{p:subgroupsofF} to analyze the groups with $E(G)=\Phi(G)=1$, $c(G)=2$, $L'$ of odd order and $\a(G)=1$. 

First observe that if $H$ represents the unique conjugacy class of core-free subgroups of $G$ with $b(G,H) \geqs 3$, then $H$ must intersect $F(G)$ nontrivially, but it cannot be contained in $F(G)$. Furthermore, there exists a normal subgroup $J$ of $L$, as  described in parts (iii) or (iv) of Theorem \ref{t:a=0} (see Remark \ref{r:J}). Then $H=AJ$, where $A = C_{p_1} \times \cdots \times C_{p_k}$ is as described in the proof of the theorem.

Now $N_L(A) = J$ (since $AN_L(A)$ is core-free and $b(G,AN_L(A)) \geqs 3$) and every subgroup $C$ of $C_{F(G)}(J)$ must be normal in $G$ (otherwise $H \times C$ is core-free and $b(G,H \times C) \geqs 3$). Moreover, $L$ must act transitively on the set of subgroups $B \leqs [J,F(G)]$ with the property that $B$ is normalized by $J$ and the subgroups $BJ$ and $AJ$ have the same normal closure in $G$. In particular, $L/J$ acts simply transitively on the set of such subgroups. These conditions essentially describe all the groups with $\a(G)=1$ in this setting, imposing various conditions on the cyclic summands of $[J,F(G)]$. 

For example, if $J = \la y \ra$ has the properties described in part (iii) of Theorem \ref{t:a=0} and $F(G)$ has a subgroup $W = (C_p)^2$ that is imprimitive as an $L$-module,  then $F(G) = W \times C_{F(G)}(y)$ and every subgroup of $C_{F(G)}(y)$ is normal in $G$. Therefore, in case (iii) we may assume that $y$ acts as a scalar on each $V_i$ and $G$ has no imprimitive minimal normal subgroups. We leave the task of describing the precise structure of the groups with $\a(G)=1$ to the interested reader.


\begin{thebibliography}{999}

\bibitem{AB}
S. Aivazidis and A. Ballester-Bolinches, \emph{On the Frattini subgroup of a finite group}, J. Algebra \textbf{470} (2017), 254--262.

\bibitem{Asch}
M. Aschbacher, \emph{Finite group theory}, Cambridge Studies in Advanced Mathematics, vol. 10, Cambridge University Press, 1986.

\bibitem{BC}
R.F. Bailey and P.J. Cameron, \emph{Base size, metric dimension and other invariants of groups and graphs}, Bull. Lond. Math. Soc. \textbf{43} (2011),  209--242.

\bibitem{BS}
M.A. Bennett and C.M. Skinner, \emph{Ternary Diophantine equations via Galois representations and modular forms}, Canad. J. Math. \textbf{56} (2004), 23--54.

\bibitem{Magma} 
W. Bosma, J. Cannon and C. Playoust, \emph{The {\textsc{Magma}} algebra system I: The user language}, J. Symb. Comput. \textbf{24} (1997), 235--265.

\bibitem{BF}
R. Brauer and K.A. Fowler, \emph{On groups of even order}, Ann. of Math. \textbf{62} (1955), 565--583.

\bibitem{BHR}
J.N. Bray, D.F. Holt and C.M. Roney-Dougal, \emph{The maximal subgroups of the low-dimensional finite classical groups}, London Math. Soc. Lecture Note Series, vol. 407, Cambridge University Press, 2013.

\bibitem{B_spor}
T.C. Burness, \emph{On soluble subgroups of sporadic groups}, Israel J. Math., to appear.

\bibitem{B20}
T.C. Burness, \emph{Base sizes for primitive groups with soluble stabilisers}, Algebra Number Theory \textbf{15} (2021), 1755--1807.
 
\bibitem{Bur181}
T.C. Burness, \emph{Simple groups, fixed point ratios and applications}, in Local representation theory and simple groups, 267--322, EMS Ser. Lect. Math., Eur. Math. Soc., Z\"{u}rich, 2018.

\bibitem{BGS}
T.C. Burness, R.M. Guralnick and J. Saxl, \emph{On base sizes for symmetric groups}, Bull. Lond. Math. Soc. \textbf{43} (2011), 386--391.

\bibitem{BH_gen}
T.C. Burness and S. Harper, \emph{Finite groups, $2$-generation and the uniform domination number}, Israel J. Math. \textbf{239} (2020), 271--367.

\bibitem{BH}
T.C. Burness and H.Y. Huang, \emph{On the Saxl graphs of primitive groups with soluble stabilisers}, Algebr. Comb. \textbf{5} (2022), 1053--1087. 

\bibitem{BH2}
T.C. Burness and H.Y. Huang, \emph{On base sizes for primitive groups of product type},  J. Pure Appl. Algebra \textbf{227} (2023), 107228, 43 pp.

\bibitem{BL}
T.C. Burness and M. Lee, \emph{On the classification of extremely primitive affine groups}, Israel J. Math., to appear.

\bibitem{BLS}
T.C. Burness, M.W. Liebeck and A. Shalev, \emph{Base sizes for simple groups and a conjecture of Cameron}, Proc. Lond. Math. Soc. \textbf{98} (2009), 116--162.

\bibitem{BLN}
T.C. Burness, A. Lucchini and D. Nemmi, \emph{On the soluble graph of a finite group},  J. Combin. Theory Ser. A \textbf{194} (2023), 105708, 39 pp. 

\bibitem{BOW}
T.C. Burness, E.A. O'Brien and R.A. Wilson, \emph{Base sizes for sporadic simple groups}, Israel J. Math. \textbf{177} (2010), 307--333.

\bibitem{BTh}
T.C. Burness and A.R. Thomas, \emph{The classification of extremely primitive groups}, Int. Math. Res. Not. IMRN (2022), no.13, 10148--10248.

\bibitem{BTV}
T.C. Burness and H.P. Tong-Viet, \emph{Primitive permutation groups and derangements of prime power order}, Manuscripta Math. \textbf{150} (2016), 255--291.

\bibitem{Faw0}
J.B. Fawcett, \emph{Bases of twisted wreath products},  J. Algebra \textbf{607} (2022), 247--271.

\bibitem{Faw}
J.B. Fawcett, \emph{The base size of a primitive diagonal group}, J. Algebra \textbf{375} (2013), 302--321.

\bibitem{F1}
J.B. Fawcett, J. M\"{u}ller, E.A. O'Brien and R.A. Wilson, \emph{Regular orbits of sporadic simple groups}, J. Algebra \textbf{522} (2019), 61--79.

\bibitem{F2}
J.B. Fawcett, E.A. O'Brien and J. Saxl, \emph{Regular orbits of symmetric and alternating groups}, J. Algebra \textbf{458} (2016), 21--52.
 
\bibitem{FKL}
L. Finkelstein, D. Kleitman and T. Leighton, \emph{Applying the classification theorem for finite simple groups to minimize pin count in uniform permutation architectures}, in VLSI algorithms and architectures (Corfu, 1988), 247--256, Lecture Notes in Comput. Sci., 319, Springer, New York, 1988.

\bibitem{Gas}
W. Gasch\"{u}tz, \emph{\"{U}ber die $\Phi$-Untergruppe endlicher Gruppen},
Math. Z. \textbf{58} (1953), 160--170.

\bibitem{James}
J.P. James, \emph{Partition actions of symmetric groups and regular bipartite graphs},  Bull. London Math. Soc. \textbf{38} (2006), 224--232.

\bibitem{KL} 
P.B. Kleidman and M.W. Liebeck, \emph{The {S}ubgroup {S}tructure of the {F}inite {C}lassical {G}roups}, London Math. Soc. Lecture Note Series, vol. 129, Cambridge University Press, 1990.

\bibitem{Lee1}
M. Lee, \emph{Regular orbits of quasisimple linear groups I}, J. Algebra \textbf{586} (2021), 1122--1194.

\bibitem{Lee2}
M. Lee, \emph{Regular orbits of quasisimple linear groups II},  J. Algebra \textbf{586} (2021), 643--717.

\bibitem{LSh3}
M.W. Liebeck and A. Shalev, \emph{Bases of primitive permutation groups}, in Groups, combinatorics \& geometry (Durham, 2001), 147--154, World Sci. Publ., River Edge, NJ, 2003.

\bibitem{Suz}
M. Suzuki,  \emph{{G}roup {T}heory I}, Grundlehren der mathematischen Wissenschaften 247,  Springer-Verlag, Berlin, 1982.

\bibitem{Wilson}
R.A. Wilson, \emph{Maximal subgroups of sporadic groups}, in Finite simple groups: thirty years of the Atlas and beyond, 57--72, Contemp. Math., 694, Amer. Math. Soc., Providence, RI, 2017.

\end{thebibliography}
\end{document}